\documentclass{article}
\usepackage[utf8]{inputenc}

\usepackage[legalpaper, margin=1in]{geometry}
\usepackage{tikz}
\usepackage{hyperref}
\usetikzlibrary{calc}
\usetikzlibrary{shapes.geometric, arrows}
\usepackage{tikz-cd}
\usepackage{bm}

\usepackage{amsmath, amsthm, xparse, physics, amsfonts,verbatim}
\usepackage{amssymb}
\usepackage{youngtab}
\usepackage{ytableau}
\usepackage{multicol}
\newtheorem{theorem}{Theorem}[section]

\usepackage{graphicx, mathtools}
\usepackage{float}
\usepackage{xcolor}

\usepackage[font={small}, labelfont={small,bf}]{caption}

\newtheorem{defi}[]{Definition}[section]
\newtheorem{prop}{Proposition}[section]
\newtheorem{question}{Question}
\numberwithin{question}{section}
\newtheorem{answer}{Answer}[section]
\newtheorem{conj}[theorem]{Conjecture}

\newtheorem*{T1}{Theorem~\ref{thm:denom-deg}}

\newtheorem{remark}[theorem]{Remark}
\newtheorem{ex}{Example}
\newtheorem{claim}{Claim}[section]
\newtheorem{corr}{Corollary}[claim]

\newcommand{\partition}[1]{\small \begin{matrix}#1\end{matrix}}
\newcommand{\play}[1]{\langle {#1} \rangle}

\newcommand{\opp}{\mathrm{op}}
\newcommand{\level}{\mathrm{level}}

\newcommand{\rebs}[1]{$\mathcal{O}^{op}_{#1}$}

\usepackage{tikz-cd}
\usetikzlibrary{shapes.geometric}
\usetikzlibrary{decorations.pathreplacing}
\tikzstyle{whitebead} = [circle, ball color=white,minimum size=3mm]
\tikzstyle{blackbead} = [circle, ball color=black,minimum size=3mm]
\tikzstyle{circ} = [circle,draw=white,fill=black,minimum size=3mm,scale=0.5]
\tikzstyle{tri} = [circle,rotate=270,draw=white,fill=red, minimum size=3mm,scale=0.5]
\tikzstyle{rec} = [circle, draw=white, fill=blue, minimum size=3mm, scale = 0.5]
\tikzstyle{diam} = [diamond, draw=black, fill=white, minimum size=3mm, scale=0.33]
\tikzstyle{str} = [circle, draw=green, fill=green, minimum size=3mm, scale=0.33]
\tikzstyle{brace} = [
    thick,
    decoration={
        brace,
        mirror,
        raise=0.25cm
    },
    decorate
]

\title{Limiting behavior in growth of Bulgarian solitaire orbits}
\author{Nhung Pham\footnote{Email: \href{mailto:pham0306@umn.edu}{pham0306@umn.edu}} \\ School of Mathematics, University of Minnesota - Twin Cities \\ Minneapolis, MN 55455}
\date{}

\begin{document}

\maketitle

\begin{abstract}
    The Bulgarian Solitaire rule induces a finite dynamical system on the set of integer partitions of $n$. Brandt \cite{brandt1982cycles} characterized and counted all cycles in its recurrent set for any given $n$, with orbits parametrized by necklaces of black and white beads. However, the transient behavior within each orbit has been almost completely unknown. The only known case is when $n=\binom{k}{2}$ is a triangular number, in which case there is only one orbit. Eriksson and Jonsson \cite{eriksson2017level} gave an analysis for convergence of the structure as $k$ grows, and to what extent the limit applied to the finite case. 
    
    In this article, we generalize the convergent structure for orbits of Bulgarian Solitaire system for any $n$. For necklaces of the form $(BW)^k = BWBW\cdots$, we give the precise limit of the generating functions as $k$ grows. For other necklaces, we prove that the generating functions are rational and provide a bound for their denominator and numerator degrees.
\end{abstract}

\vfill

\tableofcontents
\vfill

\pagebreak

\section{Introduction and Notation}\label{sec:1}

\subsection{The Bulgarian Solitaire system}

 The game of Bulgarian Solitaire (BS) was introduced by Martin Gardner in 1983. The original game starts with $45$ cards divided into a number of piles. Now keep repeating the \emph{Bulgarian Solitaire moves}: in each turn, take one card from each pile and form a new pile. The game ends when the sizes of the piles are not changed by performing the moves. Surprisingly, it turns out that regardless of initial state of the game, it must end in a finite number of moves at the state with one pile of one card, one pile of two cards, \ldots, and one pile of nine cards. The rule was then generalized for any $n$ as the Bulgarian Solitaire operation $\beta$ on the set of partition $\mathcal{P}(n) \coloneqq \{ \lambda = (\lambda_1, \lambda_2, \ldots, \lambda_m) : \text{length } l(\lambda) = m, \lambda_1 \ge \lambda_2 \ge \ldots \ge \lambda_m > 0 \text{ are integers}, \lambda_1 + \ldots + \lambda_m = n\}$. The operation $\lambda \rightarrow \beta(\lambda)$ is described as: in each step, take one from each part, form a new part and put the parts in weakly decreasing order. In addition, the game can be described in terms of Young diagrams: in each turn, remove the longest column and reinsert it as a new row into the diagram. An example is shown in Figure~\ref{fig:BS-ex}.

\begin{figure}[H]
    \centering
    \begin{tikzpicture}
    \node[] at (0, 0) {\ytableausetup{centertableaux}
        \ytableaushort
        {\none,\none,\none}
        * {5,2,2}
        * [*(lightgray)]{1,1,1}};
    \draw[->, black, thick] (2, 0) -- (3.5, 0);
    \node[] at (2.75, 0.25) {$\beta$};
    \node[] at (5, 0) {
         \ytableausetup{centertableaux}
        \ytableaushort
        {\none,\none,\none}
        * {4,3,1,1}
        * [*(lightgray)]{0,3}
        };
         
    \end{tikzpicture}
    \caption{Bulgarian Solitaire on Young diagram $\beta((5,2,2)) = (4,3,1,1)$}
    \label{fig:BS-ex}
\end{figure}
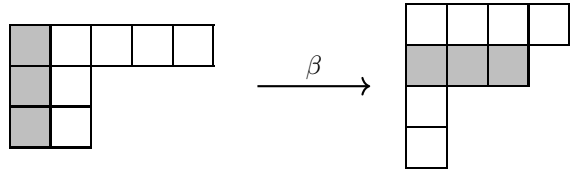

The \emph{game graph} of Bulgarian Solitaire system is a directed graph whose vertices are partitions of $n$ with directed edges connecting $\lambda \rightarrow \beta(\lambda)$. Some examples of BS \emph{game graph} are given in Figure~\ref{fig:BS-6} and Figure~\ref{fig:BS-8}. 

    \begin{figure}[H]
        \centering
        \begin{tikzpicture}
            \node[label=270:{$({3,2,1})$}] (1) at (0,11) {};
            \node[label={$({4,2})$}] (2) at (0,9) {};
            \node[label=$({3,1,1,1})$] (3) at (-1.5,8) {};
            \node[label=$({2,2,2})$] (4) at (-1.5,7) {};
            \node[label=$({3,3})$] (5) at (-1.5,6) {};
            \node[label=$({4,1,1})$] (6) at (-1.5,5) {};
            \node[label=$({2,2,1,1})$] (7) at (-1.5,4) {};
            \node[label=$({5,1})$] (8) at (1.5,8) {};
            \node[label=$({2,1,1,1,1})$] (9) at (0.5,7) {};
            \node[label=$({6})$] (10) at (2.5,7) {};
            \node[label=$({1,1,1,1,1,1})$] (11) at (2.5,6) {};
            \draw[->] (0,9.55)--(0,10.2);
            \draw[->] (-1.5,8.65)--(-0.25,9);
            \draw[->] (-1.5,7.65)--(-1.5,8.2);
            \draw[->] (-1.5,6.65)--(-1.5,7.2);
            \draw[->] (-1.5,5.65)--(-1.5,6.2);
            \draw[->] (-1.5,4.65)--(-1.5,5.2);
            \draw[->] (1.5,8.65)--(0.25,9);
            \draw[->] (0.5,7.65)--(1.25,8.2);
            \draw[->] (2.5,7.65)--(1.75,8.2);
            \draw[->] (2.5,6.65)--(2.5,7.2);
            \path[->, thick, black, -{Latex[length=2mm]}]
            (1) edge [out=135, in=45, loop] (1);
        \end{tikzpicture}
        \caption{Bulgarian Solitaire game graph for $n = 6.$}
        \label{fig:BS-6}
    \end{figure}
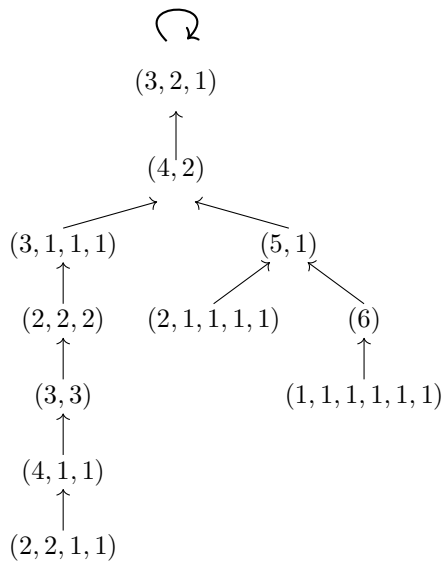
    
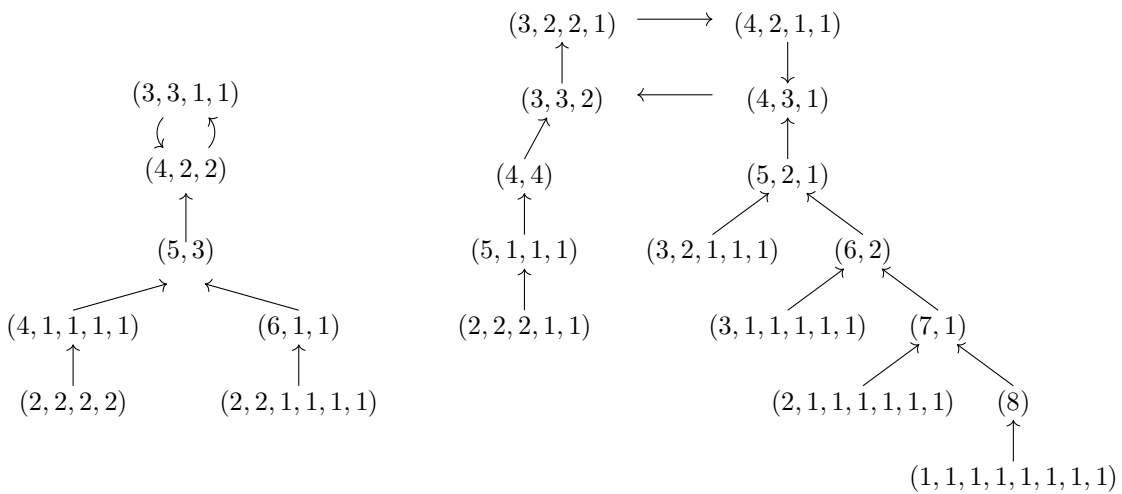
\begin{figure}[hbt]
    \centering
    \begin{tikzpicture}
        \node[] (1) at (0,10.5) {$({3,3,1,1})$};
        \node[] (2) at (0,9.5) {$({4,2,2})$};
        \node[label={$({5,3})$}] at (0,8) {};
        \node[label=$({4,1,1,1,1})$] at (-1.5,7) {};
        \node[label=$({2,2,2,2})$] at (-1.5,6) {};
        \node[label=$({6,1,1})$] at (1.5,7) {};
        \node[label=$({2,2,1,1,1,1})$] at (1.5,6) {};
        \node[label={$(3,2,2,1)$}] at (5, 11) {};
        \node[label=$({4,2,1,1})$] at (8,11) {};
        \node[label=$({4,3,1})$] at (8,10) {};
        \node[label=$({3,3,2})$] at (5,10) {};
        \node[label=$({4,4})$] at (4.5,9) {};
        \node[label=$({5,1,1,1})$] at (4.5,8) {};
        \node[label=$({2,2,2,1,1})$] at (4.5,7) {};
        \node[label=$({5,2,1})$] at (8,9) {};
        \node[label=$({3,2,1,1,1})$] at (7,8) {};
        \node[label=$({6,2})$] at (9,8) {};
        \node[label=$({3,1,1,1,1,1})$] at (8,7) {};
        \node[label=$({7,1})$] at (10,7) {};
        \node[label=$({2,1,1,1,1,1,1})$] at (9,6) {};
        \node[label=$({8})$] at (11,6) {};
        \node[label=$({1,1,1,1,1,1,1,1})$] at (11,5) {};
        \draw[->] (1) to [out=225, in=135] (2);
        \draw[->] (2) to [out=45, in=-45] (1);
        \draw[->] (0,8.55)--(0,9.2);
        \draw[->] (-1.5,7.65)--(-0.25,8);
        \draw[->] (-1.5,6.65)--(-1.5,7.2);
        \draw[->] (1.5,6.65)--(1.5,7.2);
        \draw[->] (1.5,7.65)--(0.25,8);
        \draw[->] (6,11.5)--(7,11.5);
        \draw[->] (8,11.2)--(8,10.65);
        \draw[->] (7,10.5)--(6,10.5);
        \draw[->] (5,10.65)--(5,11.2);
        \draw[->] (4.5,9.65)--(4.8,10.2);
        \draw[->] (4.5,8.65)--(4.5,9.2);
        \draw[->] (4.5,7.65)--(4.5,8.2);
        \draw[->] (8,9.65)--(8,10.2);
        \draw[->] (7,8.65)--(7.75,9.2);
        \draw[->] (9,8.65)--(8.25,9.2);
        \draw[->] (8,7.65)--(8.75,8.2);
        \draw[->] (10,7.65)--(9.25,8.2);
        \draw[->] (9,6.65)--(9.75,7.2);
        \draw[->] (11,6.65)--(10.25,7.2);
        \draw[->] (11,5.65)--(11,6.2);
        \end{tikzpicture}
        \caption{Bulgarian Solitaire game graph for $n = 8.$}
        \label{fig:BS-8}
    \end{figure}

 As in any finite dynamical system, that is, any self-map $\beta: X \rightarrow X$ on a finite set $X$, the elements in any orbit $\mathcal{O}$ under repeated application of $\beta$ eventually lead to a {\it recurrent cycle} $\bm{C} \subset \mathcal{O}$, consisting of the elements $\lambda$ in $\mathcal{O}$ having
$\lambda=\beta^m (\lambda)$ for some $m \geq 1$.  When $\beta$ is the Bulgarian solitaire map, these recurrent cycles were analyzed by Brandt \cite{brandt1982cycles} in terms of objects called necklaces. A \emph{necklace} $N$ of black and white beads is an equivalence class of sequences of letters $\{B,W\}$ under cyclic rotation. We call $N$ a \emph{primitive necklace} if it cannot be written as a concatenation $N=P^k=PP\cdots P$ of copies of another necklace $P$. We will reserve $P$ for \emph{primitive necklaces}. We also say a binary sequence \emph{encodes} a necklace class $N$ if the sequence represents an element of the class $N$ when we assign $1$ for $B$ and $0$ for $W$.
Brandt \cite{brandt1982cycles} showed there is a bijection
\begin{equation}
    \label{Brandt-map}
    \begin{aligned}
        \mathcal{O}: \text{ necklaces } &\longrightarrow \text{ BS orbits } \\
        N &\longmapsto \mathcal{O}_N
    \end{aligned}
\end{equation}
that maps a necklace to the orbit of Bulgarian Solitaire system which has the unique recurrent cycle $\bm{C}_N$ represented by the necklace. Specifically, if the necklace $N$ is of length $m$, a partition $\lambda$ is in the corresponding recurrent cycle $\bm{C}_N$ if its \emph{difference labelling} from the staircase $\Delta_{m-1} = (m-1, \ldots, 1, 0)$, defined as
\begin{equation}\label{eq:diff-label}
    \lambda^{-} = \lambda - \Delta_{m-1},
\end{equation}
encodes the necklace class $N$. Figure~\ref{fig:prim-necklace-bs} and Figure~\ref{fig:nonprim-necklace-bs} visualize the bijection.

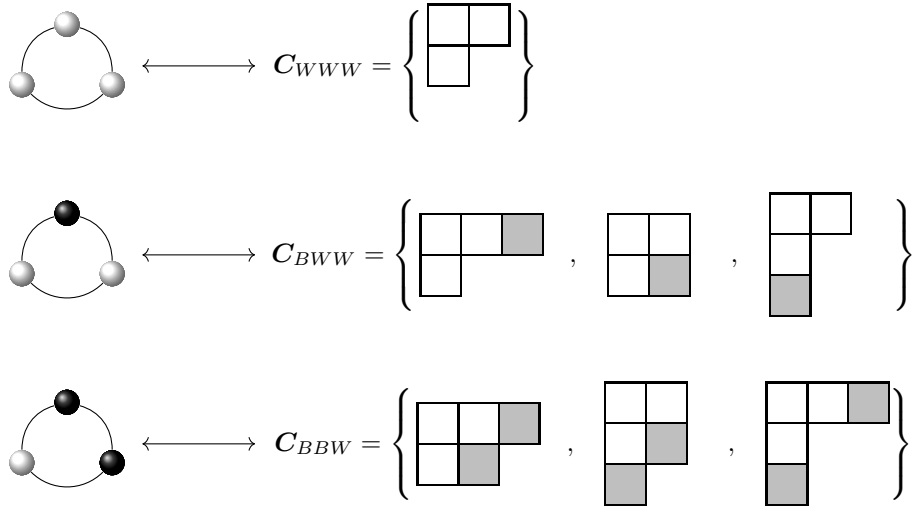
\begin{figure}[hbt]
    \centering
    \begin{tikzpicture}
        \begin{scope}[scale=1,yshift=0cm,xshift=0cm]
            \node[whitebead] (1) at (0,0.8) {};
            \node[whitebead] (2) at (0.6,0) {};
            \node[whitebead] (3) at (-0.6,0) {};
            \draw[-] (1) to [out=345,in=90] (2);
            \draw[-] (2) to [out=225,in=-45] (3);
            \draw[-] (3) to [out=90,in=195] (1);
            
            \draw[<->] (1, 0.25)--(2.5, 0.25);
            \node[label={}] at (4.5, 0.25) {$\bm{C}_{WWW} = \left\{\,\ytableausetup{centertableaux}
            \ytableaushort
            {\none,\none,\none}
            * {2,1}\, \right\}$};
        \end{scope}
        
        \begin{scope}[scale=1,yshift=-2.5cm,xshift=0cm]
            \node[blackbead] (1) at (0,0.8) {};
            \node[whitebead] (2) at (0.6,0) {};
            \node[whitebead] (3) at (-0.6,0) {};
            \draw[-] (1) to [out=345,in=90] (2);
            \draw[-] (2) to [out=225,in=-45] (3);
            \draw[-] (3) to [out=90,in=195] (1);
            
            \draw[<->] (1, 0.25)--(2.5, 0.25);
            \node[label={}] at (7, 0.25) {$\bm{C}_{BWW} = \left\{\,\ydiagram[*(lightgray)]
            {2+1,1+0}
            *[*(white)]{2,1}\quad,
            \quad\ydiagram[*(lightgray)]
            {2+0,1+1}
            *[*(white)]{2,1},
            \quad\ydiagram[*(lightgray)]
            {2+0,1+0,0+1}
            *[*(white)]{2,1}\right\}$};
        \end{scope}
        
        \begin{scope}[scale=1,yshift=-5cm,xshift=0cm]
            \node[blackbead] (1) at (0,0.8) {};
            \node[blackbead] (2) at (0.6,0) {};
            \node[whitebead] (3) at (-0.6,0) {};
            \draw[-] (1) to [out=345,in=90] (2);
            \draw[-] (2) to [out=225,in=-45] (3);
            \draw[-] (3) to [out=90,in=195] (1);
            
            \draw[<->] (1, 0.25)--(2.5, 0.25);
            \node[label={}] at (7, 0.25) {$\bm{C}_{BBW} = \left\{\,\ydiagram[*(lightgray)]
            {2+1,1+1}
            *[*(white)]{2,1}\quad,
            \quad\ydiagram[*(lightgray)]
            {2+0,1+1,0+1}
            *[*(white)]{2,1},
            \quad\ydiagram[*(lightgray)]
            {2+1,1+0,0+1}
            *[*(white)]{2,1}\right\}$};
        \end{scope}
    
\end{tikzpicture}
    \caption{The map $\mathcal{O}$ for necklaces of length $3$, which are $WWW$ (non-primitive) and $BWW, BBW$ (primitive).}
    \label{fig:prim-necklace-bs}
\end{figure}

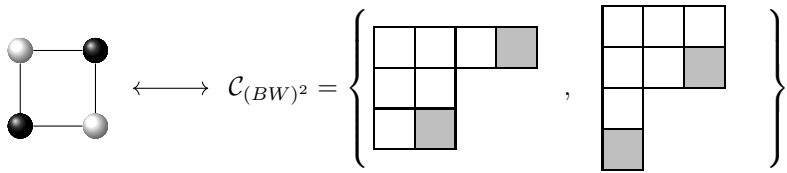
\begin{figure}[hbt]
    \centering
    \begin{tikzpicture}
        \begin{scope}[scale=1,yshift=0cm,xshift=0cm]     \node[blackbead] (1) at (0,0) {};
            \node[whitebead] (2) at (0,1) {};
            \node[blackbead] (3) at (1,1) {};
            \node[whitebead] (4) at (1,0) {};
            \draw[-] (1)--(2);
            \draw[-] (2)--(3);
            \draw[-] (3)--(4);
            \draw[-] (4)--(1);
            
            \draw[<->] (1.5, 0.5)--(2.5, 0.5);
            \node[label={}] at (6.5, 0.5) {$\mathcal{C}_{(BW)^2} = \left\{\,\ydiagram[*(lightgray)]
            {3+1,2+0,1+1}
            *[*(white)]{3,2,1}\quad,
            \quad\ydiagram[*(lightgray)]
            {3+0,2+1,1+0,0+1}
            *[*(white)]{3,2,1}\right\}$};
        \end{scope}
        
\end{tikzpicture}
    \caption{The map $\mathcal{O}$ for non-primitive necklaces of length $4$. The recurrent set in $\mathcal{O}_{(BW)^2}$} has only $2$ elements, shown above.
    \label{fig:nonprim-necklace-bs}
\end{figure}

We will also be interested in the {\it non-recurrent} elements $\lambda$ in $\mathcal{O}_N$,
and the distribution of the following {\it level} 
statistic on the orbit:
$$
\level(\lambda):=\min\{m \in \{0,1,2,\ldots\}: \beta^m(\lambda) \in \bm{C}_N \}.
$$
Our main results concern its generating function, which is defined as the following polynomial in $\mathbb{Z}[x]$:
\begin{equation*}
    \mathcal{D}_N(x):=\sum_{\lambda \in \mathcal{O}_N} x^{\level(\lambda)}.
\end{equation*}

 Note that Brandt's bijection implies that the Bulgarian solitaire system on $\mathcal{P}(n)$ for $n = \binom{k+1}{2}$ has only one orbit $\mathcal{O}_{W^{k+1}}$ with the recurrent set $\bm{C}_{W^{k+1}} = \{ \Delta_{k} \}$. Thus the game graph $\mathcal{O}_{W^{k+1}}$ turns out to be a tree, rooted at the vertex $\Delta_{k}$, and for any partition $\lambda \in \mathcal{O}_{W^{k+1}}$, the statistic $\level(\lambda)$ is the distance in the tree from $\lambda$ to the staircase $\Delta_{k}$. Eriksson and Jonsson prove \cite{eriksson2017level} that, in the limit as $k$ grows, the sequence of level sizes of $\mathcal{O}_{W^{k}}$ converges to the subsequence of evenly-indexed Fibonacci numbers $(F_{2d})_{d=0}^\infty$, with the generating function
\begin{equation}
\label{Eriksson-Jonsson-generating-function}
  \lim_{k \rightarrow \infty} \mathcal{D}_{W^k}(x) =   \dfrac{(1-x)^2}{1-3x+x^2} =: H_{W}(x).
\end{equation}
Eriksson and Jonsson also showed that for $\mathcal{O}_{W^{k+1}}$, the sizes of levels $0,1, \ldots, \lfloor \frac{k}{2} \rfloor$ in the Bulgarian solitaire game tree coincide with those of an object that they called the {\it quasi-infinite game tree} after pruning an appropriate branch (described later in this section). 

We generalize this, describing the limit of the level sizes for arbitrary $n$, with these two main results:

\begin{theorem}\label{thm:BW-limit-level}
    There is a power series $H_{BW}(x)$ in $\mathbb{Z}[[x]]$ such that
     \begin{equation*} 
     \lim_{k \rightarrow \infty} \mathcal{D}_{(BW)^k} = H_{BW}(x).
     \end{equation*}
    Moreover, $H_{BW}(x)$ is a rational function, given by
    \begin{equation*}
    \begin{aligned}
        H_{BW}(x) &= \dfrac{(x-1)^2(3x+2)}{x^3 - 3x^2 - x + 1}\\
        &= 2 + x + 3x^2 + 7x^3 + 15x^4 + 33x^5 + 71x^6 + 155x^7 + 335x^8 + \ldots
    \end{aligned}
    \end{equation*}
\end{theorem}

\begin{theorem}\label{thm:denom-deg}
    For primitive necklaces $P$ with $\abs{P} \geq 3$, there is a power series $H_P(x)$ in $\mathbb{Z}[[x]]$ such that the sequence of generating functions
    $(\mathcal{D}_{P^k})_{k=0}^\infty$
    converges to $H_P(x)$ coefficient-wise. Moreover, $H_P(x)$ is a rational function having denominator polynomial of degree at most $\abs{P}$ and numerator degree at most $2 \abs{P}$. 
\end{theorem}

\subsection{The reversed Bulgarian Solitaire digraphs}

\begin{defi}[reversed BS rule] \rm
The forward Bulgarian solitaire map $\lambda \mapsto \beta(\lambda)$ has (partially defined) reverse maps, 
which we will denote $\lambda \xrightarrow{j} R_j(\lambda)$,
for which $\beta(R_j(\lambda))=\lambda$.  They are described
as follows:
    \begin{itemize}
        \item For Young diagrams: take out the $j{th}$ row if it is no shorter than the number of rows minus $1$ and insert it again as the leftmost column.
        \item For a partition $\lambda$: take a part $\lambda_j \geq l(\lambda)-1$, then distribute it into other parts, one for each.
    \end{itemize}
\end{defi}

For example $(4,2,2) \xrightarrow{1} (3,3,1,1)$ and $(4,2,2) \xrightarrow{2} (5, 3)$. \textbf{Figure~\ref{fig:tableaux-reverseBS}} illustrates the reverse rule.

\begin{figure}[hbt]
    \centering
    \begin{tikzpicture}
        \node[] at (0,0) {\ytableausetup{centertableaux}
            \ytableaushort
            {\none,\none,\none}
            * {4,2,2}
            * [*(lightgray)]{4}
            * [*(gray)]{0,2}};
        \draw[->] (1.8, 0.35)--(3, 1);
        \node[] at (2.25, 1) {$1$};
        \draw[->] (1.8, -0.35)--(3, -1);
        \node[] at (2.25, -1) {$2$};
        \node[] at (4.5,1.5) {\ytableausetup{centertableaux}
            \ytableaushort
            {\none,\none,\none}
            * {3,3,1,1}
            * [*(lightgray)]{1,1,1,1}};
        \node[] at (5,-1.5) {\ytableausetup{centertableaux}
            \ytableaushort
            {\none,\none,\none}
            * {5,3}
            * [*(gray)]{1,1}};
    \end{tikzpicture}
    \caption{Reverse BS on a Young diagram}
    \label{fig:tableaux-reverseBS}
\end{figure}
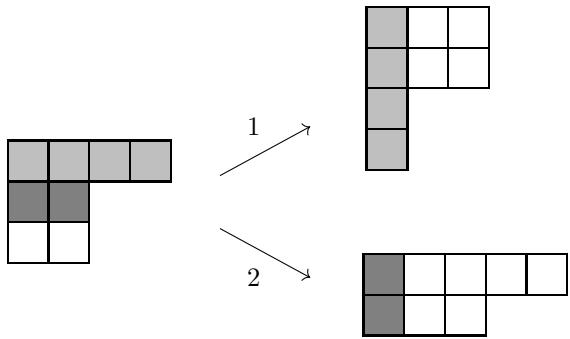

 We also use $\mathbf{0}$ as a result of an invalid move, that is, the removed row is too short to be the leftmost (longest) column. Obviously, if $\lambda$ is a partition, $R_j(\lambda)$ is valid (or legal) if and only if $\lambda_j \ge l(\lambda) -1$ and we display $\lambda_j$ in brackets, as $\langle \lambda_j \rangle$. If $\lambda \in \mathcal{O}_N$ where $N$ is a necklace of length $m$, then its \emph{difference labelling} $\lambda^-$ is defined the same as in \eqref{eq:diff-label}. Also, if $\lambda_j$ is bracketed, so is $(\lambda^-)_j$.
\begin{ex}[bracketing rule]\label{ex:diff-label}
    In $\mathcal{O}_{W^4}$ in Figure~\ref{fig:BS-6} we have 
    $$ \left( \partition{\play{3}\\ \play{2} \\ 1 } \right)^- = \partition{\play{0} \\ \play{0} \\ 0 \\ 0} \text{ and } \left( \partition{\play{5}\\ \play{1}} \right)^- = \partition{\play{2} \\ \play{-1} \\ -1 \\ 0}. $$
    In the left orbit $\mathcal{O}_{(BW)^2}$ of Figure~\ref{fig:BS-8} we have
    $$ \left( \partition{\play{4}\\ \play{2} \\ \play{2} } \right)^- = \partition{\play{1} \\ \play{0} \\ \play{1} \\ 0} \text{ and } \left( \partition{\play{4}\\ 1 \\ 1 \\ 1 } \right)^- = \partition{\play{1} \\ -1 \\ 0 \\ 1}. $$
\end{ex} 

 Note that a staircase has its last part to be $0$ while a partition requires all parts to be positive. Also, a \emph{difference labelling} does not necessarily have positive parts. The inverse is denoted 
    $$\mu^+ = \mu + \Delta_{m-1},$$ 
where $\mu$ is a difference labelling in the orbit $\mathcal{O}_N$ and $\mu^+$ must be an ordinary partition. For example, the inverses for difference labelings in Example~\ref{ex:diff-label} are
    $$ \left( \partition{\play{0} \\ \play{0} \\ 0 \\ 0}\right)^+ = \partition{\play{3}\\ \play{2}\\ 1} \,\text{ and } \left( \partition{\play{2}\\ \play{-1}\\-1 \\ 0} \right)^+ = \partition{\play{5} \\ \play{1}}.$$

 The reversed BS rule reverses all the arrows in $\mathcal{O}_N$ for some necklace $N$ to get the digraph $\mathcal{O}_N^\opp$. A \emph{(reverse BS) playing sequence} $\sigma$ from $\lambda$ is a sequence of parts that are consecutively played legally starting from $\lambda$. We use $R_\sigma(\lambda)$ to represent the result of performing $\sigma$ from $\lambda$.
\begin{ex}[(reverse BS) playing sequence] \rm 
    When reversing all the arrows in Figure~\ref{fig:BS-6} and Figure~\ref{fig:BS-8}, we get \rebs{W^4}, \rebs{(BW)^2} and \rebs{BBWW}, respectively. A playing sequence from $(3,2,1)$ in \rebs{W^4} is $[121]$ and $R_{[121]}((3,2,1)) = (2,1,1,1,1)$. A playing sequence from $(3,3,1,1)$ in \rebs{(BW)^2} can be $[1111\ldots]$, where $R_{[1^{2t}]}((3,3,1,1)) = (3,3,1,1)$ and $R_{[1^{2t+1}]}((3,3,1,1)) = (4,2,2)$ for integer $t$.
\end{ex}

For any primitive necklace $P$ of length $p$, there are $p$ elements in the recurrent set $\bm{C}_{P^k}$ for any $k$. We consider the difference labellings of the elements in the recurrent set $\bm{C}_P$, which are $\{0,1\}$-vectors of length $p$ that encode necklace class $P$. Their playable parts are bracketed following specific rules given in Section~\ref{sec:3}. Let that set be
    $$\mathcal{C}_P = \{ \gamma^{(t)} \in \bm{C}_P : R_1(\gamma^{(t)}) = \gamma^{(t+1)}, 1 \le t \le p \text{ and } \gamma^{(p+1)} \equiv \gamma^{(1)} \}.$$

\begin{ex}
\begin{equation*}
    \mathcal{C}_{BWW} = \{ BWW, WWB, WBW \} = \Bigg\{ \partition{\play{1} \\ \play{0} \\ 0}, \partition{\play{0}\\ 0\\1}, \partition{\play{0}\\ \play{1}\\0} \Bigg\}. 
\end{equation*}
\end{ex}

  For some vector $\alpha$, we will use $\alpha^k$ for the concatenation $\partition{\alpha \\ \vdots \\ \alpha}$ of $k$ blocks $\alpha$. Brandt's bijection \cite{brandt1982cycles} maps necklace $P^k$ to the orbit whose recurrent cycle $\mathcal{C}_{P^k} = \{ (\gamma^{(t)})^k : \gamma^{(t)} \in \mathcal{C}_P \}$.  We also use 
    $$ \mathcal{C}_{P^k}^+ := \{ \alpha^+ : \alpha \in \mathcal{C}_{P^k}\}$$
for the ordinary recurrent set (whose elements are the same as those of $\bm{C}_{P^k}$ but their playable parts are bracketed). The bracketing rule for this set will be discussed in Section~\ref{sec:3}.

  With the definition of the reversed rule in hand, we now discuss  Eriksson and Jonsson's \emph{quasi-infinite game tree} $\mathcal{F}_{W}$ for orbits \rebs{W^k}. \textbf{Figure~\ref{fig:reverse-BS-finite}} display some initial difference reversed BS game graphs up to some levels and \textbf{Figure~\ref{fig:inf-game-tree}} is the \emph{quasi-infinite game tree} $\mathcal{F}_W$. After pruning the branch formed by playing sequences $[1\ldots]$ and adding a self-cycle to the root, we can see the coincidence between the \emph{quasi-infinite tree} and the finite trees up to some levels.

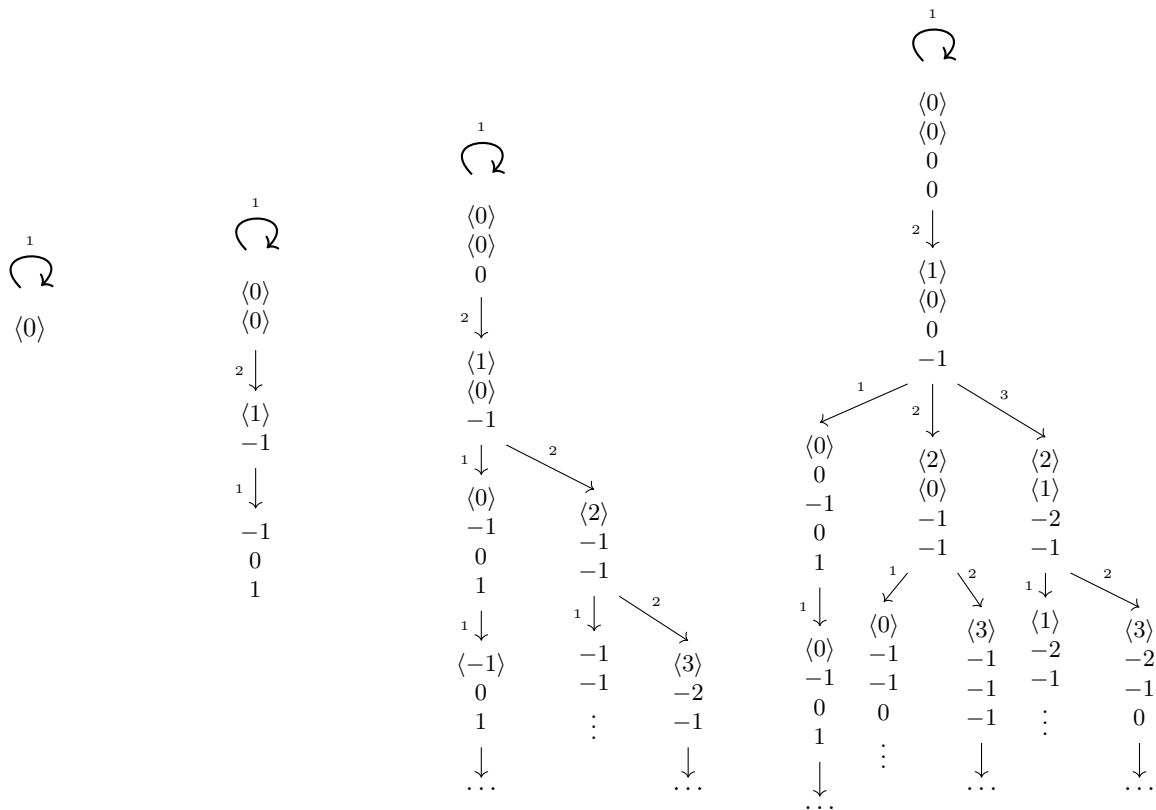
\begin{figure}[H]
    \centering
    \begin{tikzpicture}
    \begin{scope}[scale = 1, xshift=0cm, yshift=-0.5cm]
            \node[label=270:{$\play{0}$}] (1) at (0,11) {};
            \path[<-, thick, black, -{Latex[length=2mm]}]
            (1) edge [out=135, in=45, loop] (1);
            \node[] at (0, 11.75) {\tiny $1$};
    \end{scope}
    
    \begin{scope}[scale=1, xshift = 3cm, yshift = 0cm]
            \node[label=270:{$\partition{\play{0}\\ \play{0}}$}] (1) at (0,11) {};
            \node[] (2) at (0,8.75) {$\partition{\play{1}\\-1}$};
            \node[] (3) at (0,7) {$\partition{-1\\0\\1}$};
            \draw[<-] (2)--(0,9.8) node [pos = 0.5, left = 1pt] {\tiny $2$};
            \draw[<-] (3)--(2) node [pos = 0.5, left = 1pt] {\tiny $1$};
            \path[<-, thick, black, -{Latex[length=2mm]}]
            (1) edge [out=135, in=45, loop] (1);
            \node[] at (0, 11.75) {\tiny $1$};
    \end{scope}
    
    \begin{scope}[scale=1, xshift = 6cm, yshift = 0cm]
            \node[label=270:{$\partition{\play{0}\\ \play{0}\\ 0}$}] (1) at (0,12) {};
            \node[] (2) at (0,9.25) {$\partition{\play{1}\\ \play{0} \\ -1}$};
            \node[] (3) at (0,7.25) {$\partition{\play{0}\\ -1 \\ 0 \\ 1}$};
            \node[] (4) at (0,5.25) {$\partition{\play{-1} \\ 0 \\ 1}$};
            \node[] (5) at (0, 4) {$\ldots$};
            \node[] (8) at (1.5,7.25) {$\partition{\play{2} \\ -1 \\ -1}$};
            \node[] (9) at (1.5,5.25) {$\partition{-1\\ -1 \\ \vdots}$};
            \node[] (10) at (2.75,5.25) {$\partition{\play{3}\\-2\\-1}$};
            \node[] (11) at (2.75, 4) {$\ldots$};
            \draw[<-] (2)--(0,10.5) node [pos = 0.5, left = 1pt] {\tiny $2$};
            \draw[<-] (3)--(2) node [pos = 0.5, left = 1pt] {\tiny $1$};
            \draw[<-] (4)--(3) node [pos = 0.5, left = 1pt] {\tiny $1$};
            \draw[<-] (5)--(4);
            \draw[<-] (8.north)--(2.south east) node [pos = 0.5, right = 1pt, above = 1pt] {\tiny $2$};
            \draw[<-] (9.north)--(8) node [pos = 0.5, left = 1pt] {\tiny $1$};
            \draw[<-] (10.north)--(8.south east) node [pos = 0.5, right = 1pt, above = 1pt] {\tiny $2$};
            \draw[<-] (11)--(10);
            \path[<-, thick, black, -{Latex[length=2mm]}]
            (1) edge [out=135, in=45, loop] (1);
            \node[] at (0, 12.75) {\tiny $1$};
        \end{scope}
        
        \begin{scope}[scale=1, xshift = 12cm, yshift = -0.5cm]
            \node[label=270:{$\partition{\play{0}\\ \play{0}\\ 0 \\ 0}$}] (1) at (0,14) {};
            \node[] (2) at (0,10.75) {$\partition{\play{1}\\ \play{0} \\ 0 \\ -1}$};
            \node[] (3) at (-1.5,8.25) {$\partition{\play{0}\\ 0\\ -1 \\ 0 \\ 1}$};
            \node[] (4) at (-1.5,5.75) {$\partition{\play{0} \\-1 \\ 0 \\ 1}$};
            \node[] (5) at (-1.5, 4.25) {$\ldots$};
            \node[] (6) at (0,8.25) {$\partition{\play{2}\\ \play{0}\\ -1 \\ -1}$};
            \node[] (6a) at (-0.65,5.75) {$\partition{\play{0}\\ -1 \\ -1 \\ 0 \\ \vdots}$};
            \node[] (6b) at (0.65,6) {$\partition{\play{3} \\ -1 \\ -1 \\ -1}$};
            \node[] (7) at (0.65, 4.5) {$\ldots$};
            \node[] (8) at (1.5,8.25) {$\partition{\play{2} \\ \play{1}\\ -2 \\ -1}$};
            \node[] (9) at (1.5,6) {$\partition{\play{1}\\ -2 \\ -1 \\ \vdots}$};
            \node[] (10) at (2.75,6) {$\partition{\play{3}\\-2\\-1\\0}$};
            \node[] (11) at (2.75, 4.5) {$\ldots$};
            \draw[<-] (2)--(0,12.15) node [pos = 0.5, left = 1pt] {\tiny $2$};
            \draw[<-] (3.north)--(2.south west) node [pos = 0.5, left = 1pt, above=1pt] {\tiny $1$};
            \draw[<-] (4)--(3) node [pos = 0.5, left = 1pt] {\tiny $1$};
            \draw[<-] (5)--(4);
            \draw[<-] (6)--(2) node [pos = 0.5, left = 1pt] {\tiny $2$};
            \draw[<-] (6a.north)--(6.south west) node [pos = 0.5, left = 1pt, above=1pt] {\tiny $1$};
            \draw[<-] (6b.north)--(6.south east) node [pos = 0.5, right = 1pt, above=1pt] {\tiny $2$};
            \draw[<-] (7)--(6b);
            \draw[<-] (8.north)--(2.south east) node [pos = 0.5, right = 1pt, above = 1pt] {\tiny $3$};
            \draw[<-] (9.north)--(8) node [pos = 0.5, left = 1pt] {\tiny $1$};
            \draw[<-] (10.north)--(8.south east) node [pos = 0.5, right = 1pt, above = 1pt] {\tiny $2$};
            \draw[<-] (11)--(10);
            \path[<-, thick, black, -{Latex[length=2mm]}]
            (1) edge [out=135, in=45, loop] (1);
            \node[] at (0, 14.75) {\tiny $1$};
        \end{scope}
    \end{tikzpicture}
    \caption{Difference reversed BS game graph for $n = \binom{k+1}{2}$ with $k = 1, 2, 3, 4$ up to level $\lfloor \frac{k+1}{2} \rfloor + 1$.}
    \label{fig:reverse-BS-finite}
\end{figure}

\begin{figure}[hbt]
    \centering
    \begin{tikzpicture}
    \node[] at (4, 0) {$\partition{\play{0} \\ \play{0}}$};
        \draw[->] (3.85, -0.75)--(3, -1.25);
        \node[] at (3.2, -0.9) {\tiny $1$};
        \node[] at (2.75, -2) {$\partition{\play{0} \\ \play{0}}$};
        \draw[->] (2.75, -3)--(2.75, -3.5);
        \node[] at (2.75, -4) {$\ldots$};
        \draw[->] (4.15, -0.75)--(5, -1.25);
        \node[] at (4.8, -0.9) {\tiny $2$};
        \node[] at (5, -2.25) {$\partition{\play{1} \\ \play{0} \\ \play{0}}$};
        \draw[->] (4.85, -3.25)--(4, -3.75);
        \node[] at (4.15, -3.45) {\tiny $1$};
        \node[] at (4, -4.7) {$\partition{\play{0} \\ 0}$};
        \draw[->] (4, -5.5)--(4,-6);
        \node[] at (4, -6.7) {$\partition{\play{0} \\ \play{0}}$};
        \draw[->] (4, -7.25)--(4, -7.75);
        \node[] at (4, -8.25) {$\ldots$};
        \draw[->] (5, -3.25)--(5, -3.75);
        \node[] at (4.85, -3.45) {\tiny $2$};
        \node[] at (5, -4.7) {$\partition{\play{2} \\ \play{0} \\ \play{0}}$};
        \draw[->] (4.85, -5.5)--(4.45, -6);
        \node[] at (4.5, -5.75) {\tiny $1$};
        \node[] at (4.5, -6.7) {$\partition{0 \\ 0}$};
        \draw[->] (5.15, -5.5)--(5.55, -6);
        \node[] at (5.15, -5.75) {\tiny $2$};
        \draw[->] (5, -5.5)--(5, -6);
        \node[] at (5, -6.85) {$\partition{\play{3} \\ \play{0} \\ \play{0}}$};
        \draw[->] (5, -7.5)--(5, -8);
        \node[] at (5, -8.5) {$\ldots$};
        \node[] at (5.65, -5.75) {\tiny $3$};
        \node[] at (5.65, -7) {$\partition{\play{3} \\ \play{1} \\ \play{0}\\ \play{0}}$};
        \draw[->] (5.65, -8)--(5.65, -8.5);
        \node[] at (5.65, -9) {$\ldots$};
        \draw[->] (5.15, -3.25)--(7.15, -3.75);
        \node[] at (6.5, -3.45) {\tiny $3$};
        \node[] at (7.25, -5) {$\partition{\play{2} \\ \play{1} \\ \play{0} \\ \play{0}}$};
        \draw[->] (7.15, -6.1)--(6.6, -6.6);
        \node[] at (6.6, -6.25) {\tiny $1$};
        \node[] at (6.6, -7.55) {$\partition{\play{1} \\ 0\\0}$};
        \draw[->] (6.6, -8.35)--(6.6, -8.85);
        \node[] at (6.6, -9.15) {$\ldots$};
        \draw[->] (7.25, -6.1)--(7.25, -6.6);
        \node[] at (7.15, -6.35) {\tiny $2$};
        \node[] at (7.25, -7.55) {$\partition{\play{3}\\ \play{0} \\ \play{0}}$};
        \draw[->] (7.25, -8.35)--(7.25, -8.85);
        \node[] at (7.25, -9.15) {$\ldots$};
        \draw[->] (7.35, -6.1)--(8, -6.6);
        \node[] at (7.85, -6.35) {\tiny $3$};
        \node[] at (8, -7.55) {$\partition{\play{3}\\ \play{2}\\ \play{0} \\ \play{0}}$};
        \draw[->] (8, -8.5)--(8, -9);
        \node[] at (8, -9.25) {$\ldots$};
        \draw[->] (7.5, -6.1)--(8.75, -6.6);
        \node[] at (8.75, -6.35) {\tiny $4$};
        \node[] at (8.75, -7.75) {$\partition{\play{3}\\ \play{2}\\ \play{1}\\ \play{0} \\ \play{0}}$};
        \draw[->] (8.75, -8.75)--(8.75, -9.25);
        \node[] at (8.75, -9.5) {$\ldots$};
    \end{tikzpicture}
    \caption{The \emph{quasi-infinite game tree} $\mathcal{F}_W$.}
    \label{fig:inf-game-tree}
\end{figure}
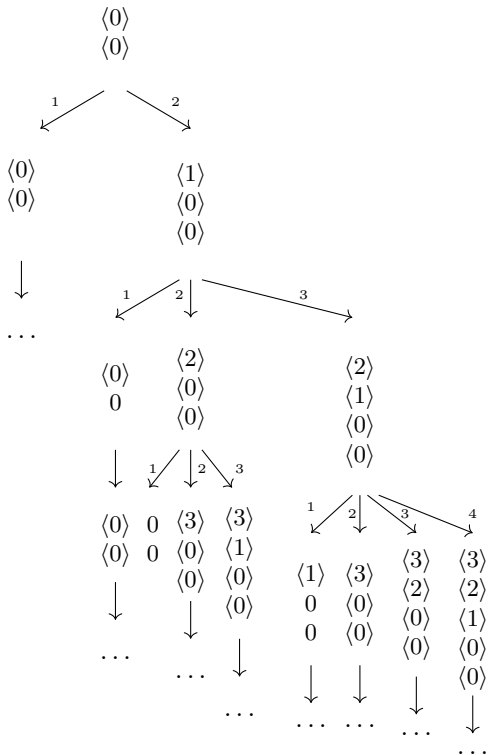
  Recall that $\mathcal{C}_{W^k} = \{ (0)^k \}$. $\mathcal{F}_W$ starts at $\partition{\play{0}\\ \play{0}}$ which represents all bracketed parts in the roots of any trees \rebs{W^k} for $k \ge 2$. The rules for $\xrightarrow{i}$ in the \emph{quasi-infinite game tree} \cite[\S3]{eriksson2017level} are described below:
  
\begin{enumerate}
    \item Delete the bracketed $i$th part.
    \item Increase all parts above it by $1$ and make them bracketed.
    \item Bracket the new $i$th part (if there is one) if it differs at most $1$ from the old one.
    \item If a zero was played, add zeros at the end so that there are two, and make them bracketed.
\end{enumerate}
Eriksson and Jonsson \cite{eriksson2017level} showed
\begin{equation*}
    \mathcal{O}^\opp_{W^\infty}=\lim_{m \rightarrow \infty} \mathcal{O}^\opp_{W^m}.
\end{equation*}

  Generalizing that idea, we will build a \emph{P-quasi-infinite forest} $\mathcal{F}_P$ consists of $t$ trees denoted $G_1, \ldots, G_p$, with $G_t$ rooted at $\gamma^{(t)} \in \mathcal{C}_P$. The remaining vertices of $G_t$ are generated from $\gamma^{(t)}$ by applying the reverse BS operations $R_i$ in all possible ways, and modifying the parts according to rules described in the next section. An example is the $BWWW$-quasi-infinite forest in Figure~\ref{fig:BWWW-inf-graph}. 

\begin{figure}[hbt]
    \centering
    \begin{tikzpicture}
        \node[] at (0, 1.25) {$G_1$};
        \node[] (0) at (0,0) {$\boldsymbol{\partition{\play{0}\\ 0 \\ 0 \\ 1}}$};
        \node[] (1) at (0, -2) {$\boldsymbol{\partition{\play{0} \\ \play{0} \\ \play{1} \\ 0}}$};
        \draw[->] (0)--(1) node[pos=0.5, left=1pt] {\tiny $1$};
        \node[] (2) at (0, -3.75) {$\ldots$};
        \draw[->] (1)--(2);
        
        \node[] at (3, 1.25) {$G_2$};
        \node[] (3) at (3, 0.25) {$\boldsymbol{\partition{\play{0} \\ \play{0} \\ \play{1} \\ 0}}$};
        \node[] (4) at (2.25, -2) {$\boldsymbol{\partition{\play{0} \\ \play{1} \\ 0\\0}}$};
        \draw[->] (3.south west)--(4.north) node[pos=0.5, left=1pt] at (5.2, -0.85) {\tiny $1$};
        \node[] (5) at (2.25, -3.75) {$\ldots$};
        \draw[->] (4)--(5);
        \node[] (6) at (3.75, -2.25) {$\partition{\play{1}\\ \bm{\play{1}} \\ \bm{\play{0}} \\ \bm{0}\\ \bm{0}}$};
        \draw[->] (3.south east)--(6.north) node[pos=0.5, right=1pt] {\tiny $2 \slash 3$};
        \node[] (7) at (3, -4.8) {$\partition{\play{1} \\ 0\\0 \\ 0}$};
        \draw[->] (6.south west)--(7.north) node[pos=0.5, left=1pt] {\tiny $1$};
        \node[] (8) at (3, -7) {$\boldsymbol{\partition{\play{0} \\ 0\\0 \\ 1}}$};
        \draw[->] (7)--(8) node[pos=0.5, left=1pt] {\tiny $1$};
        \node[] (9) at (3.75, -5) {$\partition{\play{2} \\ \bm{\play{0}} \\ \bm{0} \\ \bm{0} \\ \bm{1}}$};
        \draw[->] (6)--(9.north) node[pos=0.5, right=2] at (8, -3) {\tiny $2$};
        \node[] at (3.75, -6.5) {$\ldots$};
        \node[] (10) at (4.5, -5) {$\partition{\play{2}\\ \play{2} \\ \bm{\play{0}} \\ \bm{\play{0}}\\ \bm{\play{1}} \\ \bm{0}}$};
        \draw[->] (6.south east)--(10.north) node[pos=0.5, right=2] at (8, -3) {\tiny $3$};
        \node[] at (4.5, -6.5) {$\ldots$};
        
        \node[] at (6, 1.25) {$G_3$};
        \node[] (11) at (6,0.35) {$\boldsymbol{\partition{\play{0} \\ \play{1}\\0\\0}}$};
        \node[] (12) at (6, -2.25) {$\boldsymbol{\partition{\play{1} \\ \play{0} \\ 0 \\ 0}}$};
        \draw[->] (11)--(12) node[pos=0.5, left=1pt] {\tiny $1 \slash 2$};
        \node[] (13) at (6, -3.75) {$\ldots$};
        \draw[->] (12)--(13);

        \node[] at (9, 1.25) {$G_4$};
        \node[] (14) at (9, 0.25) {$\boldsymbol{\partition{\play{1} \\ \play{0} \\ 0 \\ 0}}$};
        \node[] (15) at (8.25, -2.25) {$\boldsymbol{\partition{\play{0} \\ 0\\0\\1}}$};
        \draw[->] (14.south west)--(15.north) node[pos=0.5, left=1pt] {\tiny $1$};
        \node[] (16) at (8.25, -3.75) {$\ldots$};
        \draw[->] (15)--(16);
        \node[] (17) at (9.75, -2.25) {$\partition{\play{2} \\ \bm{\play{0}} \\ \bm{\play{0}}\\ \bm{\play{1}}\\ \bm{0}}$};
        \draw[->] (14.south east)--(17.north) node[pos=0.5, right=1pt] {\tiny $2$};
        \node[] (18) at (9, -5) {$\partition{0\\0 \\ 1\\0}$};
        \draw[->] (17.south west)--(18.north) node[pos=0.5, left=1pt] {\tiny $1$};
        \node[] (19) at (9.75, -5) {$\partition{\play{3} \\ \bm{\play{0}} \\ \bm{\play{1}} \\ \bm{0} \\ \bm{0}}$};
        \draw[->] (17)--(19) node[pos=0.5, left=1pt]{\tiny $2$};
        \node[] at (9.75, -6.5) {$\ldots$};
        \node[] (20) at (10.5, -5.25) {$\partition{\play{3} \\ \play{1} \\ \bm{\play{1}} \\ \bm{\play{0}} \\ \bm{0}\\ \bm{0}}$};
        \draw[->] (17.south east)--(20.north) node[pos=0.5, right=1pt]{\tiny $3$};
        \node[] at (10.5, -6.5) {$\ldots$};
        
    \end{tikzpicture}
    \caption{BWWW quasi-infinite forest. The bold entries are $\gamma^{(t)}$ for some $t$.}
    \label{fig:BWWW-inf-graph}
\end{figure}
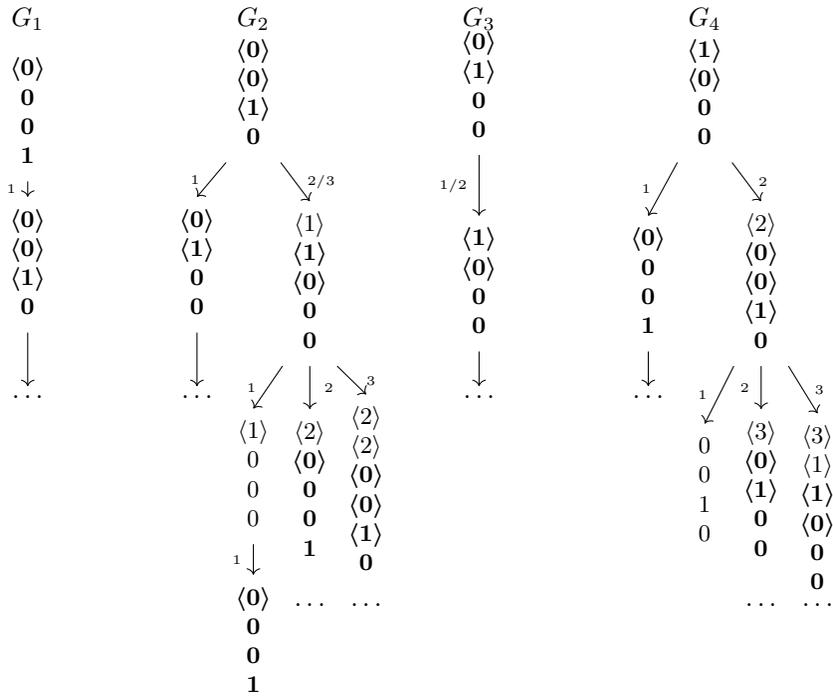

  Let $d(\lambda, \gamma^{(t)})$ be the length of the reversed playing sequence from $\gamma^{(t)}$ to $\lambda$. Let $g_t(x)$ be the generating functions by level sizes (the growth function) of $G_t$, that is, 
$$
g_t(x):=\sum_{\lambda \in G_t} x^{d(\lambda,\gamma^{(t)})}.
$$

  We will show later that 
\begin{center}
\rebs{P^\infty} $= \lim_{k \rightarrow \infty}$ \rebs{P^k} 
\end{center}
is the quasi-infinite forest $\mathcal{F}_P$ after pruning the branches formed by playing sequences $[1 \ldots]$ and adding arrows $\gamma^{(t)} \xrightarrow{1} \gamma^{(t+1)}$ to make the roots a recurrent cycle. In fact, up to level $k$, the finite digraph \rebs{P^k} coincides with the limit digraph \rebs{P^\infty}.

\section{The rules for producing the forest \texorpdfstring{$\mathcal{F}_P$}{forest}}\label{sec:rule}

The idea for the \emph{P-quasi-infinite forest} is that each root $\gamma^{(t)}$ is actually an infinite periodic binary vector whose consecutive segments $(\gamma^{(t)}_{mp+1}, \ldots, \gamma^{(t)}_{(m+1)p})$ of length $p$ each forms a copy of $\gamma^{(t)}$. The reason for this beginning will be explained in the next section. This infinite vector guarantees that we will not encounter negative parts in difference labellings when performing the $R_j$ operations as in the finite cases. However, we cannot write out infinite-length vectors for the elements in the forest. Instead, we start with only one block and add more blocks as one goes further down the tree $G_t$ and applies the operators $R_j$. In fact, each segment $(\gamma^{(t)}_{mp+s}, \ldots, \gamma^{(t)}_{(m+1)p+s-1})$ is a copy of the root $\gamma^{(t+s)}$ for any nonegative integers $m, s$.  

In this section, for a vector $\lambda = (\lambda_1, \ldots, \lambda_m)$ and $j \le k$, we write $\lambda[j:k] = (\lambda_j, \ldots, \lambda_k)$. Also denote $\lambda_j = \lambda[1:j]$ and $\lambda[j:] = (\lambda_j, \ldots, \lambda_m)$.

As part of the rules, we will want to maintain the property that every $\lambda$ in the forest $\mathcal{F}_P$ has a \emph{tail} (= final segment of length $p$) exactly matching one of the roots $\gamma^{(t)}$, although the subscript $t$ may not match the subscript $t'$ of the tree $G_{t'}$ which contains $\lambda$.  Assume in the rules below that $\lambda$'s tail matches the root $\gamma^{(t)}$.
\vskip .1in
\noindent
\textit{Rules for $\lambda \xrightarrow{i} R_i(\lambda)$ in the tree $G_t$ in general quasi-infinite forest - that is, $\lambda_i$ is bracketed and $R_i$ is played:}
\begin{enumerate}
    \item If $\lambda_i$ is in the tail (among the lowest $p$ parts), append to $\lambda$ so that $\lambda[i:] = \gamma^{(s)}$ for some (unique) $s$. Replace $\lambda[i:]$ by $\gamma^{(s+1)}$ and bracket the tail as in the root $\gamma^{(s+1)}$. 
    \item Otherwise, delete $\lambda_i$. If $\partition{\lambda_i \\ \lambda_{i+1}} = \partition{s \\ s+1}$, playing either $R_i$ or $R_{i+1}$ has the same result, so we label the play by $\xrightarrow{i \slash i+1}$ and apply $R_{i+1}$. Bracket the new $i^{th}$ part (if there is one) if it differs at most $1$ from $\lambda_i$. If there are two consecutive entries $\partition{s\\s+1}$ and $s$ is bracketed, so is $s+1$.
    \item Increase all entries $\lambda[:(i-1)]$ by $1$ each, and bracket them.
\end{enumerate}

Here is an example of these rules in the $BWWW$ quasi-infinite forest from Figure~\ref{fig:BWWW-inf-graph}.
\begin{ex}\label{ex:step-gen-rule}
    Let $\lambda = \partition{\play{2}\\ \bm{\play{0}}\\ \bm{\play{0}}\\ \bm{\play{1}}\\ \bm{0}}$. The bold parts are $\gamma^{(2)}$.
    \begin{itemize}
        \item Play $\lambda \xrightarrow{3} R_3(\lambda)$:
        \begin{enumerate}
            \item[Step 1:] Note that $\lambda_3$ is among the last $4$ parts, we append to $\;
                \lambda \longrightarrow \partition{\play{2}\\ \play{0}\\ \bm{\play{0}}\\ \bm{\play{1}}\\ \bm{0} \\ \bm{0}}\,,\,$
            so that $\lambda[3:] = \gamma^{(3)}$. 
            \item[Step 2:] Replace the tail $\gamma^{(3)}$ by $\gamma^{(4)}$ and bracket as in $\gamma^{(4)}$: $\;
                \partition{\play{2}\\ \play{0}\\ \bm{\play{0}}\\  \bm{\play{1}}\\ \bm{0} \\ \bm{0}} \longrightarrow \partition{\play{2}\\ \play{0}\\ \bm{\play{1}}\\ \bm{\play{0}}\\ \bm{0} \\ \bm{0}}.$
            \item[Step 3:] Add $1$ to each parts in $\lambda[:2]$ to get 
            $ R_3(\lambda) = \partition{\play{3}\\ \play{1}\\ \bm{\play{1}}\\ \bm{\play{0}}\\ \bm{0} \\ \bm{0}}.$
        \end{enumerate}
        \item Play $\lambda \xrightarrow{1} R_1(\lambda)$. Note that $\lambda_1$ is not among the last $4$ parts, so rule $2$ applies to obtain 
        $ R_1(\lambda) = \partition{0 \\ 0 \\ 1 \\ 0}\;,$
        because the new first part, $0$, differs from $\lambda_1$ by $2$.
    \end{itemize}
\end{ex}


\section{Limiting level sizes of \texorpdfstring{\rebs{(BW)^k}}{orb}  as \texorpdfstring{$k$}{k} grows}\label{sec:BW-limit}

\subsection{The quasi-infinite forest \texorpdfstring{$\mathcal{F}_{BW}$}{forest}}
  We first analyze $P=BW$. The exception in this case is that the root $\partition{\play{1}\\ \play{0}} \in \mathcal{C}_{BW}$ we defined in Section~\ref{sec:1} does not reflect all the playable parts in $\mathcal{C}_{(BW)^k}^+$ for $k \ge 2$ (e.g $\partition{\play{3}\\ \play{2}\\ \play{1}\\ 0}$). However, $\partition{\play{1}\\ \play{0}\\ \play{1}}$ does, by the same argument as in Proposition~\ref{prop:roots} in Section~\ref{sec:3}. Thus, we build the forest $\mathcal{F}_{BW}$ as in Figure~\ref{fig:BW-quasi-tree} following the rules given in the Section~\ref{sec:rule}, except that the roots are $\gamma^{(1)} = \partition{\play{0} \\ \play{1}}$ and $\gamma^{(2)} = \partition{\play{1} \\ \play{0} \\ \play{1}}$.

\begin{figure}[H]
    \centering
    \begin{tikzpicture}
        \node[] at (0, 1) {$\bm{G_1}$};
        \node[] at (0,0) {$\partition{\play{0} \\ \play{1}}$};
        \draw[->] (0, -0.75)--(0, -1.25);
        \node[] at (-0.35, -1) {\tiny $1 \slash 2$};
        \node[] at (0, -2) {$\partition{\play{1} \\ \play{0} \\ \play{1}}$};
        \draw[->] (-0.15, -2.75)--(-0.5, -3.25);
        \draw[->] (0.15, -2.75)--(0.5, -3.25);
        \node[] at (-0.5, -3.75) {$\ldots$};
        \node[] at (0.5, -3.75) {$\ldots$};
        
        \node[] at (4, 1) {$\bm{G_2}$};
        \node[] at (4, 0) {$\partition{\play{1} \\ \play{0} \\ \play{1}}$};
        \draw[->] (3.85, -0.75)--(3, -1.25);
        \node[] at (3.2, -0.9) {\tiny $1$};
        \node[] at (2.75, -2) {$\partition{\play{0} \\ \play{1}}$};
        \draw[->] (2.75, -2.55)--(2.75, -3.05);
        \node[] at (2.75, -3.5) {$\ldots$};
        \draw[->] (4.15, -0.75)--(5, -1.25);
        \node[] at (4.8, -0.9) {\tiny $2 \slash 3$};
        \node[] at (5, -2.25) {$\partition{\play{2} \\ \play{1} \\ \play{0} \\ \play{1}}$};
        \draw[->] (4.85, -3.25)--(4, -3.75);
        \node[] at (4.15, -3.45) {\tiny $1$};
        \node[] at (4, -4.7) {$\partition{\play{1} \\ 0 \\ 1}$};
        \draw[->] (4, -5.5)--(4,-6);
        \node[] at (3.85, -5.65) {\tiny $1$};
        \node[] at (4, -6.7) {$\partition{\play{0} \\ \play{1}}$};
        \draw[->] (4, -7.25)--(4, -7.75);
        \node[] at (4, -8.25) {$\ldots$};
        \draw[->] (5, -3.25)--(5, -3.75);
        \node[] at (4.85, -3.45) {\tiny $2$};
        \node[] at (5, -4.7) {$\partition{\play{3} \\ \play{0} \\ \play{1}}$};
        \draw[->] (4.85, -5.5)--(4.45, -6);
        \node[] at (4.5, -5.75) {\tiny $1$};
        \node[] at (4.5, -6.7) {$\partition{0 \\ 1}$};
        \draw[->] (5.15, -5.5)--(5.55, -6);
        \node[] at (5.65, -5.75) {\tiny $2 \slash 3$};
        \node[] at (5.65, -7) {$\partition{\play{4} \\ \play{1} \\ \play{0}\\ \play{1}}$};
        \draw[->] (5.65, -8)--(5.65, -8.5);
        \node[] at (5.65, -9) {$\ldots$};
        \draw[->] (5.15, -3.25)--(7.15, -3.75);
        \node[] at (6.5, -3.45) {\tiny $3 \slash 4$};
        \node[] at (7.25, -5) {$\partition{\play{3} \\ \play{2} \\ \play{1} \\ \play{0}\\ \play{1}}$};
        \draw[->] (7.15, -6.1)--(6.6, -6.6);
        \node[] at (6.6, -6.25) {\tiny $1$};
        \node[] at (6.6, -7.55) {$\partition{\play{2} \\ 1\\0\\1}$};
        \draw[->] (6.6, -8.35)--(6.6, -8.85);
        \node[] at (6.6, -9.15) {$\ldots$};
        \draw[->] (7.25, -6.1)--(7.25, -6.6);
        \node[] at (7.15, -6.35) {\tiny $2$};
        \node[] at (7.25, -7.55) {$\partition{\play{4}\\ \play{1} \\0\\1}$};
        \draw[->] (7.25, -8.35)--(7.25, -8.85);
        \node[] at (7.25, -9.15) {$\ldots$};
        \draw[->] (7.35, -6.1)--(8, -6.6);
        \node[] at (7.85, -6.35) {\tiny $3$};
        \node[] at (8, -7.55) {$\partition{\play{4}\\ \play{3}\\ \play{0} \\ \play{1}}$};
        \draw[->] (8, -8.5)--(8, -9);
        \node[] at (8, -9.25) {$\ldots$};
        \draw[->] (7.5, -6.1)--(8.75, -6.6);
        \node[] at (8.75, -6.35) {\tiny $4 \slash 5$};
        \node[] at (8.75, -8) {$\partition{\play{4}\\ \play{3}\\ \play{2}\\ \play{1} \\ \play{0}\\ \play{1}}$};
        \draw[->] (8.75, -9.25)--(8.75, -9.75);
        \node[] at (8.75, -10) {$\ldots$};
    \end{tikzpicture}
    \caption{$BW$-quasi-infinite  $\mathcal{F}_{BW}$.}
    \label{fig:BW-quasi-tree}
\end{figure}
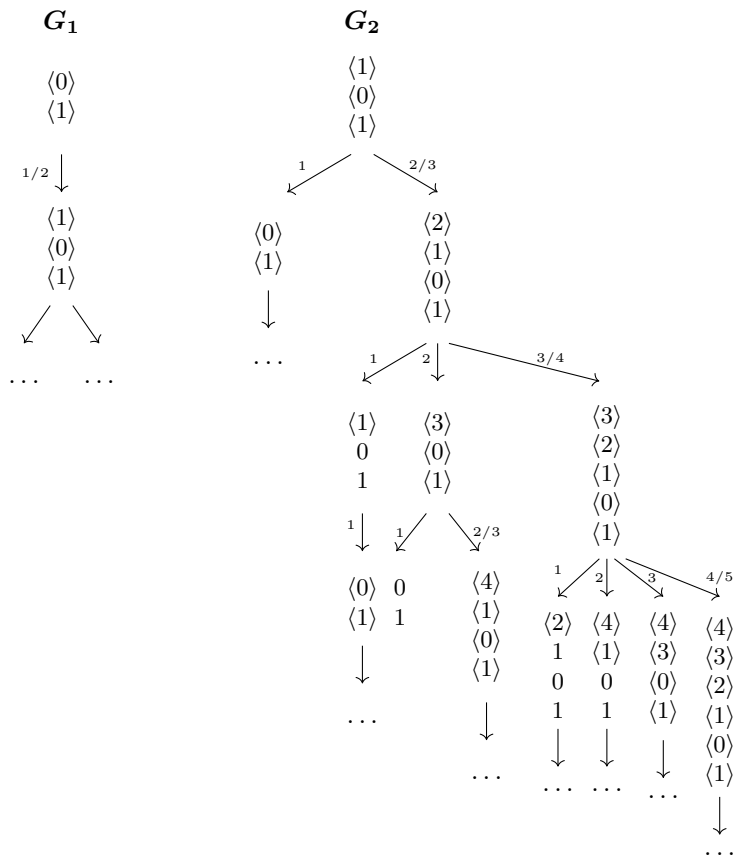

\subsection{Proof of Theorem~\ref{thm:BW-limit-level}}
  Let $g(x)$ be the generating function for the level sizes of $\mathcal{F}_{BW}$, then $g(x) = g_1(x) + g_2(x).$ Starting from $\gamma^{(1)}$, we play $R_{1/2}$ (the only playable part) and get back to the root $\gamma^{(2)}$. Thus the level generating functions
of the two trees satisfy
\begin{equation}
\label{relating-the-two-trees}
    g_1(x) = 1 + xg_2(x).
\end{equation}

  As a reminder, if we have $\partition{\play{0}\\ \play{1}}$ in a difference $\lambda$, adding a staircase makes their values the same. Thus playing either of them results in the same element. To represent this move in a playing sequence, we will use the vertex with $\play{0}$ on it. Moreover, we use the notation $R_\sigma$ for playing sequences in the quasi-infinite game graph similarly to the normal BS game graph. e.g. $R_{[222]}\left( \partition{\play{1}\\ \play{0}\\ \play{1}} \right) = \partition{\play{4}\\ \play{1}\\ \play{0}\\ \play{1}}$.

  Similarly to Jonsson and Eriksson's paper \cite[Proposition~3.1, p.~4]{eriksson2017level}, we have the following proposition:

\begin{prop}\label{prop:BW-seq}
    The tree $G_2$ rooted at $\gamma^{(2)} = \partition{\play{1}\\ \play{0}\\ \play{1}}$ has the following properties:
    \begin{enumerate}
        \item[i.] Once $R_1$ is played, only $R_1$ can be played (until $\gamma^{(1)} = \partition{\play{0}\\ \play{1}}$ or a leaf is reached). 
        \item[ii.] For $r \ge 2$, the playing sequence $[234\ldots r 1^r]$ leads to  $\gamma^{(1)}$.
    \end{enumerate}
\end{prop}

\begin{proof}
    \begin{itemize}
        \item[i.] When we play $R_1$, if the second part differs from the first part by at least $2$, nothing is bracketed in the next state, thus we reach the leaf. Otherwise, according to the rule $3$, either we reach the root $\partition{\play{0}\\ \play{1}}$ if the second and third parts are $\partition{0\\1}$ or the only bracketed part in the next state is the first part.
        
        \item[ii.] It is easy to see that 
        \begin{equation*}
            R_{[23 \ldots r]} \left( { \partition{\play{1}\\ \play{0}\\ \play{1}}} \right) = { \partition{\play{r}\\ \play{r-1}\\ \vdots\\ \play{1}\\ \play{0} \\ \play{1}}}
        \end{equation*}
        and thus playing $[1^r]$ consecutively deletes the first $r$ rows and reaches {\small $\partition{\play{0}\\ \play{1}}$}. From property $i$, we see that this is the only way to get back to the roots.
    \end{itemize}
\end{proof}

  Now we give the proof of \textbf{Theorem~\ref{thm:BW-limit-level}}:

We begin by using 
Proposition~\ref{prop:BW-seq} to construct the \emph{growth function}, or the generating function of $g_2$ by its level sizes. Table~\ref{tab:BW-quasi} shows various playing sequences and their contributions to the generating function
$g_2(x)$, explained below. We use $[\ge t]$ to denote the set of playing sequence with entries no less than $t$, including the empty sequence.

\begin{table}[H]
    \centering
    \begin{tabular}{c c |c c }
        \textbf{Sequence} & \textbf{Contribution} & \textbf{Sequence} & \textbf{Contribution} \\
        \hline
        $[1 \ldots]$ & $xg_1$ & $[2[ \ge 2]]$ & $xg_2$\\
        $[21^2 \ldots]$ & $x^3g_1$ & $[2[ \ge 2]1]$ & $x^2g_2$  \\
        $[231^3 \ldots]$ & $x^5g_1$ & $[23[ \ge 3]1^2]$ & $x^4g_2$\\
        $[2341^4 \ldots]$ & $x^7g_1$ & $[234[\ge 4]1^3 ]$ & $x^6g_2$\\
    \end{tabular}
    \caption{Growth function for tree $G_2$.}
    \label{tab:BW-quasi}
\end{table}

  By Proposition~\ref{prop:BW-seq}.ii., the sequences $[23\ldots r1^r]$ lead to the root of $G_1$, so each of them contribute the whole $G_1$ tree at level $2r-1$, that is, $x^{2r-1}g_1(x)$. This is also true for playing sequence $[1]$, since $R_1 \left( \partition{\play{1}\\ \play{0}\\ \play{1}} \right) =  \partition{\play{0}\\ \play{1}}$. 

Next, if we play $R_2$, we reach $\partition{\play{2}\\ \play{1}\\ \play{0}\\ \play{1}} = \partition{\play{2}\\ \gamma_2}$. Thus if we leave the top part untouched, which means we only play parts of indices greater or equal than $2$, then we have a subtree that is isomorphic to $G_2$. The isomorphism is defined by excluding the top part. Thus sequences $[2[\ge 2]]$ contribute $xg_2$. Similarly with $[23\ldots r [\ge r]1^{r-1}]$, each contributes $x^{2r-2}g_2$, since 
\begin{equation*}
    R_{[23\ldots r]} \left( \gamma_2 \right) = \partition{\play{r}\\ \play{r-1}\\ \vdots\\ \play{1}\\ \play{0} \\ \play{1}} = \partition{\play{r}\\ \play{r-1}\\ \vdots\\ \gamma_2}.
\end{equation*}

  If we play $[23\ldots r [\ge r]]$, the top $r-1$ rows above {\tiny $\partition{\play{1}\\ \play{0}\\ \play{1}}$} are always playable due to rule $2$. Thus the playing sequences $[23\ldots r [\ge r]1^s]$ with any $s \le r-1$ are legal. For each $r \ge 2$, we only count for $[23\ldots r [\ge r]1^r]$, since if $s < r$, $[23\ldots r [\ge r]1^s]$ are counted as $[23\ldots s [\ge s]1^s]$. Hence, the type $[23\ldots r [\ge r]1^{r-1}]$ contributes $x^{2r-2}g_2$. 

  Therefore, we obtain
\begin{align*}
    g_2(x) &= 1 + (x + x^3 + \ldots) g_1(x) + (x + x^2 + x^4 + \ldots) g_2(x) \\
    &= 1 + (x + x^3 + \ldots) + (x + 2x^2 + 2x^4 + \ldots)g_2(x) \qquad\text{(substituting for }g_1(x)\text{ using \eqref{relating-the-two-trees}})\\
    &= \left( 1 + \dfrac{x}{1-x^2} \right) + x g_2(x) + \dfrac{2x^2}{1 - x^2} g_2(x) 
\end{align*}
Bring all occurrences of $g_2(x)$ to the left side gives
$$ \dfrac{x^3 - 3x^2 - x + 1}{1-x^2} g_2(x) = \dfrac{-x^2+x+1}{1-x^2}$$
and therefore
$$g_2(x) = \dfrac{-x^2+x+1}{x^3 - 3x^2 - x + 1}.$$

  From this one concludes, again using \eqref{relating-the-two-trees}, that
$$ 
\begin{aligned}
g(x) &= g_1(x) + g_2(x)
= (1+x g_2(x)) + g_2(x)
= 1 + (1 + x) g_2(x)\\
&= \dfrac{-3x^2+x+2}{x^3-3x^2-x+1}
= \dfrac{(1-x)(3x+2)}{x^3-3x^2-x+1}. 
\end{aligned}
$$

  However, we desire the generating function for the level sizes of \rebs{(BW)^k} in the limit as $k \rightarrow \infty$. As constructed, our quasi-infinite forest has an entire copy of itself after playing $R_1$, giving rise to the left branch $[1\ldots]$, which we wish to disregard. Letting $H_{BW}(x)$ denote the height generating function for the rest
of the quasi-infinite forest, that is, the two roots and the elements in the branch $[2\ldots]$, one then has
$$ g(x) = xg(x) + H_{BW}(x)$$
and therefore
$$ H_{BW}(x) = (1-x)g(x)=\dfrac{(1-x)^2(3x+2)}{x^3 - 3x^2 - x + 1}.
$$
Thus to complete the proof of 
\textbf{Theorem~\ref{thm:BW-limit-level}}, 
it only remains to show that the level sizes of \rebs{(BW)^k} actually converge to the coefficients given by $H_{BW}(x)$.  
This is a consequence of a somewhat more precise theorem.

\begin{theorem}\label{thm:fin-coincide}
    The finite reverse Bulgarian solitaire graph \rebs{(BW)^{k+1}} coincides at least up to level $k$ with the BW quasi-infinite forest after removing its branch $[1\ldots]$.
\end{theorem}

\begin{proof}
    For each $k$, to get to level $k$ in the quasi-infinite forest, there are at most $k$ times that we change the tail $\partition{\play{0}\\ \play{1}}$ to $\partition{\play{1}\\ \play{0}\\ \play{1}}$, that means we add at most $k$ blocks $\partition{\play{0}\\ \play{1}}$. Thus if we start with the root $\partition{\play{1}\\ \play{0}\\ \play{1} \\ {\quad\,\left(\partition{0 \\ 1} \right)^{k-1}}\\ 0} \;$, we don't need to add new parts until using up those $k$ blocks of $\partition{0\\1}$. Moreover, it is easy to confirm that the rule of bracketing makes sense if $k$ is large enough. This implies the coincidence of the quasi-infinite forest with the finite graph.
\end{proof}

  Next, some kind of similar arguments apply to primitive necklaces of length at least $3$.

\section{Explanations for the rules of the general forest for primitive necklaces \texorpdfstring{$\abs{P} \ge 3$}{l}}\label{sec:3}

We wish to explain why the rules in Section~\ref{sec:rule} for constructing the (quasi-infinite) forest $\mathcal{F}_P$ from the recurrent set $\mathcal{C}_P$ actually produce a forest which is  
the limiting digraph
$
\lim_{k \rightarrow \infty} \mathcal{O}^\opp_{P^k}
$
(after pruning the branches whose playing sequences are $[1\ldots]$ and adding arrows $\gamma^{(t)} \xrightarrow{1} \gamma^{(t+1)}$ to form the recurrent cycle $\mathcal{C}_P$). To this end, we answer some below questions about the construction. In this section, if no further information is given, the statements relates to the forest. Also, the $\lambda = \partition{\mu \\ \nu}$ is used to mean the concatenation of either two partitions or two difference labelings $\mu$ and $\nu$. For a partition or a difference labeling $\lambda$ and an integer $h$, we write $\lambda + h = (\lambda_1 + h, \ldots, \lambda_m + h)$. Obviously, if $\lambda$ is a difference labeling then $(\lambda + h)^+ = \lambda^+ + h$.

\begin{question}
    How do we bracket the roots $\gamma^{(t)}$ of the tree $G_t$, even though they are infinite vectors?
\end{question}

\begin{answer} \rm We bracket them as in the recurrent set of the orbit $\mathcal{O}_P$. The following proposition clarifies that bracketing rule:

\begin{prop}\label{prop:roots}
    For any primitive necklace $P$ of length at least $3$ and any $k$, elements in the recurrent set $\mathcal{C}_{P^k}$ restricted to the first $p$ parts are identical (with brackets). Moreover, only the first $3$ parts can be bracketed.
\end{prop}

\begin{proof}
    Recall that for a partition $\lambda$ in the reverse BS system, $\lambda_j$ is playable if and only if $\lambda_j \ge l(\lambda) - 1$. If $\lambda \in \mathcal{C}_{P^k}^+$, it must have the form $\lambda = \Delta_{kp-1} + (\gamma^{(t)})^k$ for some $t$. 
    Then $l(\lambda) \ge kp-1$, and thus $\lambda_j$ is playable if and only if $\lambda_j = kp - j + (\gamma^{(t)})^k_j \ge kp - 2$, or $1 \ge (\gamma^{(t)})^k_j \ge j - 2$. Hence $j \le 3$. 
    Therefore, the bracketed parts in elements of the recurrent set $\mathcal{C}_{P^k}$ are all among the first $3 \le p$ parts, and our proposition is verified.
\end{proof}

\begin{ex} \rm
    The recurrent sets in the finite reverse BS graphs $\mathcal{O}^\opp_{BWW}$ and $\mathcal{O}^\opp_{(BWW)^2}$ are
    \begin{align*}
        \mathcal{C}_{BWW}^+ &= \Bigg\{ \partition{\play{3} \\ \play{1}}, \partition{\play{2}\\ 1\\1}, \partition{\play{2}\\ \play{2}} \Bigg\} \text{ and } \mathcal{C}_{BWW} = \Bigg\{ \partition{\play{1} \\ \play{0} \\ 0}, \partition{\play{0}\\ 0\\1}, \partition{\play{0}\\ \play{1}\\0} \Bigg\} \\
        \mathcal{C}_{(BWW)^2}^+ &= \left\{ \partition{\play{6} \\ \play{4}\\ 3 \\ 3 \\ 1}, \partition{\play{5}\\ 4\\4 \\ 2 \\ 1\\ 1}, \partition{\play{5}\\ \play{5}\\ 3 \\ 2 \\ 2} \right\}  \text{ and } \mathcal{C}_{(BWW)^2} = \left\{ \partition{\play{1} \\ \play{0} \\ 0\\ 1 \\ 0 \\ 0}, \partition{\play{0}\\ 0\\1\\ 0 \\0 \\ 1}, \partition{\play{0}\\ \play{1}\\0\\ 0 \\ 1 \\ 0} \right\}.
    \end{align*}
    Their difference labellings, including their brackets, are identical if we restrict them to the first $3$ parts.
\end{ex}

\begin{remark}\label{rmk:bracket-root} \rm Moreover, as we restrict the roots to their first $p$ parts, we only need to know how to bracket the recurrent cycle $\mathcal{C}_P$. For a partition $\lambda \in \mathcal{C}_P$, that is, $\lambda^-$ is one of the roots, we have:
\begin{itemize}
    \item $\lambda_1$ is always bracketed, because $\lambda_1 \ge p - 1 \ge l(\lambda) - 1$.
    \item $\lambda_2$:
        \begin{itemize}
            \item If $\lambda^-_2 = 1$ then it is bracketed, because then $\lambda_2 = (p-2)+1 = p-1 \ge l(\lambda)-1$.
            \item If $\lambda^-_2 = 0$ then it is bracketed only if $\lambda^-_p = 0$, since then $\lambda_2 = p-2 = l(\lambda) - 1$. 
        \end{itemize}
    \item $\lambda_3$ can only be bracketed if $\lambda^-_3 = 1$ and $\lambda^-_p = 0$, since then $\lambda_3 = (p-3) + 1 = p-2 = l(\lambda)-1$. 
\end{itemize}
\end{remark}

\end{answer}

\begin{question}
    In the forest, how do we decide which parts should be bracketed after playing a part?
\end{question}

\begin{answer} \rm Recall that due to the rules of the forest, each element has its length $p$ tail matching $\gamma^{(t)}$ for some $t$. 
There are two cases for what we play:
\begin{itemize}
\item[(1)] Play a part in the tail $\gamma^{(t)}$
\item[(2)] Play a part above the tail $\gamma^{(t)}$. 
\end{itemize}
Before discussing the bracketings, we prove a general property:

\begin{claim}\label{claim:no-3-inc}
    In a difference labelling $\lambda$, there are no triples $j_1<j_2<j_3$ such that $\lambda_{j_1} < \lambda_{j_2} < \lambda_{j_3}$. This assertion is still true for any finite \rebs{P^k} where $k \ge j_3$.
\end{claim}

\begin{proof}
    Suppose there is such triple, then $\lambda_{j_3} \ge 2$, which means some part lower than $\lambda_{j_3}$ has already been played. But during any play, the amounts added to $\lambda_{j_1}$ and $\lambda_{j_2}$ are at least the amount added to $\lambda_{j_3}$. Since $\lambda_{j_3}$ was either $0$ or $1$ in the root of the tree containing $\lambda$, $\lambda_{j_1}$ was negative at the start. We obtain a contradiction. 
\end{proof}

\begin{corr}\label{corr:diff-not-far}
    A similar argument to the proof above shows that if $\lambda_j < \lambda_{j+1}$ then $\lambda_{j+1} - \lambda_{j} = 1$.
\end{corr}

  Now we explain how to change the bracketing in $\lambda$ in the forest in situation (2), that is, after we play a part above the tail $\gamma^{(t)}$.

\begin{prop}\label{prop:playable}
    Let $\lambda$ be a difference labelling in \rebs{N}. 
    If a part $\lambda_j$ is playable, let $\psi = R_j(\lambda)$. Then $\psi_{j}$ is playable if and only if $\abs{\lambda_j - \lambda_{j+1}} \le 1$. Moreover, if $\lambda_j \ge \lambda_{j+1}$, any parts $\psi_s$ where $s \ge j +3$ are not playable. Otherwise $\lambda_j + 1 = \lambda_{j+1}$ and any parts $\psi_s$ where $s \ge j +4$ are not playable.
\end{prop}

\begin{proof}
    For the first statement, we consider $\Lambda = \lambda^+$.
    Since playing parts other than $\Lambda_j$ and $\Lambda_{j+1}$ does not affect the gap between them, we only consider the moment when performing $R_j$. The result is a partition $\overline{\Lambda}$ with $\Lambda_j$ parts. Thus, a part is playable in the state $R_j(\lambda)$ if and only if its size is at least $\Lambda_j - 1$. Because $R_j$ adds one to every other part, and the difference labeling is obtained by subtracting a staircase, which yields $\lambda_{j}-\lambda_{j+1} = (\Lambda_j - \Lambda_{j+1}) - 1$, we obtain the first statement.

    Now if $\lambda_j \ge \lambda_{j+1}$, let $\abs{N} = m$ where we subtract $\Delta_{m-1}$ to get the difference labellings. Then $\psi = R_j(\lambda)$ has $m-j+\lambda_j$ parts. Note that $\psi_{j+3} = \lambda_{j+4}$. By Claim~\ref{claim:no-3-inc}, we see that $\lambda_{j+4} \le \max \{\lambda_s : j \le s \le j+3\}$. By Corollary~\ref{corr:diff-not-far}, that maximum is at most $\lambda_j + 1$. Thus $$\psi^+_{j+3} = m -j-3+\lambda_{j+4} \le m-2-j+\lambda_j = l(\psi) - 2.$$ Thus $\psi_{j+3}$ is not playable and neither are the later parts. Moreover, $\psi^+_{j+1} = m-j-1+\lambda_{j+2} \ge m-j+\lambda_j-1$ if and only if $\lambda_{j+2} \ge \lambda_{j}$. Similarly, $\psi_{j+2}$ is playable if and only if $\lambda_{j+3} \ge \lambda_{j}+1$. Hence, $\lambda_{j+3} = \lambda_j + 1$.
    
    Otherwise $\lambda_{j+1} - \lambda_{j} = 1$, adding a staircase we get $\lambda^+_j = \lambda^+_{j+1}$ and thus we can play $R_{j+1}$ instead of $R_j$.
\end{proof}

Proposition~\ref{prop:playable} explains the rule $2$ of the forest. Moving on, we explain how a part in the ``tail" $\gamma^t$ is played, provided that the multiple $k$ (in the necklace $N=P^k$) is large enough. 

\begin{prop}\label{prop:play-tail}
    Assume that $\lambda = \partition{\alpha\\ \gamma^{(t)} \\ \gamma^{(t)} \\ \beta}$ is a difference labelling in \rebs{P^k}, where $\alpha, \beta$ are some difference labellings, that $l(\alpha) = a$ and that $\lambda_{a+1}$ is playable.  Then 
    \begin{equation}\label{eq:play-tail}
        R_{a+1}(\lambda) = \partition{\alpha + 1 \\ \gamma^{(t+1)} \\ \gamma^{(t)}[2:] \\ \beta'}
    \end{equation}
    and the $\gamma^{(t+1)}$ segment is bracketed similarly to the root.
\end{prop}

\begin{proof}
   We have $l(\gamma^{(t)}) = p$ and note that $\beta$ is not necessarily non-negative. Let $n = pk-1$. The partition corresponding to $\lambda$ is
    \begin{equation*}
        \Lambda = \lambda^+ = \lambda + \Delta_{n}
    \end{equation*}
    of length $l \ge a + 2p$. So $\lambda_{a+1}$ is playable if and only if$\Lambda_{a+1} = \gamma^{(t)}_1 + n - a \ge l - 1$, then $1 \ge \gamma^{(t)}_1 \ge l - n + a - 1$. Playing $\lambda_{a+1}$ adds $1$ to each of the entries of both $\alpha^+$ and $\alpha$, erases $\lambda_{a+1}$ and keeps the rest of $\lambda$, thus we obtain \eqref{eq:play-tail}. Let $\Psi = R_{a+1}(\Lambda)$ and $\psi =  R_{a+1}(\lambda)$, so $l(\Psi) = \gamma^{(t)}_1 + n - a$. Now in the root $\gamma^{(t+1)}$, the part $\gamma^{(t+1)}_j$ is bracketed if and only if
    \begin{equation*}
        \gamma^{(t+1)}_j + (p-j) \ge \begin{cases}
            p - 1 \qquad \text{ if } \gamma^{(t)}_1 = 1 \\
            p - 2 \qquad \text{ if } \gamma^{(t)}_1 = 0
        \end{cases},
    \end{equation*}
    thus $\gamma^{(t+1)}_j \ge j - 2 + \gamma^{(t)}_1$. We show the same $j^{th}$ part of the $\gamma^{(t+1)}$ segment is bracketed in $\psi$. Actually, that part is $\psi_{a+j}$, and
    \begin{equation*}
        \Psi_{a+j} = \gamma^{(t+1)}_j + n - a - j + 1 \ge \gamma^{(t)}_1 + n - a - 1 = l(\Psi) - 1.
    \end{equation*}
    This concludes the proof.
\end{proof}

\begin{prop}\label{prop:ignore-upper-parts}
    Let $\lambda = R_\sigma((\gamma^{(t)})^k)$, for some $(\gamma^{(t)})^k$ in $\mathcal{C}_{P^k}$ and some arbitrary playing sequence $\sigma$ of length $l$. If $k$ is large enough, then $\lambda$ is of the form $\lambda = \partition{\alpha \\ \gamma^{(s)} \\ \gamma^{(s)} \\ \beta}$ for some $1 \le s \le p$ and $l(\alpha) + 1 \ge \sigma_l$. Moreover, $\alpha$ is determined uniquely for all large $k$ including the brackets, and if rule $1$ applies to $R_{\sigma_l}$ then the first $\gamma^{(s)}$ is bracketed similarly to the root $\gamma^{(s)}$ as in Remark~\ref{rmk:bracket-root}.
\end{prop}

\begin{proof}
    Take $k \ge l + 2$, and we start with $(\gamma^{(t)})^k = \partition{\gamma^{(t)} \\ \gamma^{(t)} \\ \beta} = \partition{\gamma^{(t)}_1 \\ \gamma^{(t+1)} \\ \gamma^{(t+1)} \\ \beta'} = \partition{\gamma^{(t)}_1\\ \gamma^{(t)}_2 \\ \gamma^{(t+2)} \\ \gamma^{(t+2)} \\ \beta''} $, in which at most $3$ first parts are playable by Proposition~\ref{prop:roots}. In each play, if rule $1$ applies, then Proposition~\ref{prop:play-tail} guarantees the next state has the same form as in the proposition because in each play we only lose at most $1$ block of $\gamma^{(s)}$ for some $s$. Otherwise, rule $2$ applies which does not change the tail (consisting of blocks of $\gamma^{(s)}$). Proposition~\ref{prop:playable} shows that rule $2$ also determines the playable parts in the next state. The conclusion follows for all trees rooted at $(\gamma^{(t)})^h$ for $h \ge k$, since $\sigma$ only takes place in the first $k$ blocks.
\end{proof}

\end{answer}

\begin{question}
    Why can we restrict the infinite vector to some finite prefix without losing any playable parts?
\end{question}

\begin{answer} \rm 
    We defined earlier that \rebs{P^\infty} $= \lim_{k \to \infty}$ \rebs{P^k}. Because of Proposition~\ref{prop:ignore-upper-parts}, we see that for a playing sequence $\sigma$ and $k$ large enough, the elements $\lambda_h = R_\sigma((\gamma^{(t)})^h)$ for $h \ge k$ share the same prefix $\lambda = \partition{\alpha \\ \gamma^{(s)}}$ as in the statement of the proposition. Thus we can represent all those $\lambda_h$'s by that prefix $\lambda$ in \rebs{P^\infty}. Actually, $\lambda$ is in the form $\partition{\alpha \\ \gamma^{(s)} \\ \gamma^{(s)} \\ \vdots}$ so that rule $1$ works for playing a part in the "tail" of $\lambda$. 
    Rule $2$ also works since each element $\lambda$ has a tail of length $p \ge 3$ and by Proposition~\ref{prop:playable}, the number of new bracketed parts is at most $3$, thus the lower parts are not affected.
\end{answer}

  The conclusion of this section is that, the finite reverse BS forest \rebs{P^k} coincides with the infinite digraph \rebs{P^\infty} obtained from the forest $\mathcal{F}_P$ (by modifying any branch of $\mathcal{F}_P$ as in Section~\ref{sec:1}) up to at least level $k$. Therefore, we confirm that \rebs{P^\infty} $= \lim_{k \to \infty}$ \rebs{P^k}.

\section{The limiting generating function for general \texorpdfstring{\rebs{P^k}}{orb} by level sizes}

We recall our second main theorem from the Introduction.

\begin{T1}
    For primitive necklaces $P$ with $\abs{P} \geq 3$, there is a power series $H_P$ in $\mathbb{Z}[[x]]$ such that the sequence of generating functions
    $(\mathcal{D}_{P^k})_{k=0}^\infty$
    converges to $H_P$ coefficient-wise. Moreover, $H_P$ is a rational function having denominator polynomial of degree at most $\abs{P}$ and numerator degree at most $2 \abs{P}$.
\end{T1}

  We also recall a property that was mentioned in Section~\ref{sec:3}:
\begin{claim}\label{claim:limit-branch}
    Each $\gamma^{(t)} \in \mathcal{C}_P$ has at most $3$ playable parts. Moreover, if $\gamma^{(t)}_p = 1$, then $\gamma^{(t)}_2$ is playable if and only if$\gamma^{(t)}_2 = 1$.
\end{claim}

  The next claim points out a ``special" element in the recurrent set.
\begin{claim}\label{claim:special-tree}
    For any primitive necklace $P$ of length at least $3$, there is at least one $\gamma^{(t)} \in \mathcal{C}_P$
    (and possibly more than one) of the form $\partition{\play{\sigma} \\ 0 \\ \vdots},$ where $\sigma \in \{0, 1 \}$.
\end{claim}

\begin{proof}
    There are two cases: 
    \begin{enumerate}
    \item[1.] $P$ has two consecutive black beads, that means any $\gamma^{(t)} \in \mathcal{C}_P$ has either two consecutive $1's$ or a $1$ at the top and a $1$ at the bottom. Let $\gamma^{(t)}$ be of the form $\,\partition{1 \\ 0 \\ \vdots \\ 1}\,$, exists or else $P = B^p$ not primitive. Thus $\gamma^{(t)}_1 + p - 1 = p > p - 1 = l(\gamma^{(t)}) - 1$, so it is playable. Also, Claim~\ref{claim:limit-branch} admits that $\gamma^{(t)}_2$ is not playable.
    
    \item[2.] $P$ has two consecutive white beads, consider $\gamma^{(t)} = \partition{0 \\0 \\ \vdots \\ 1}$. Thus $\gamma^{(t)}_1$ is playable and Claim~\ref{claim:limit-branch} admits that $\gamma^{(t)}_2$ is not playable.
    \end{enumerate}
\end{proof}

  The third claim regards a special situation that separates types of playing sequences:
\begin{claim}\label{claim:limit-deg-function}
    Let $S$ be the subtree $S$ of the forest $\mathcal{F}_P$ which is rooted at the element 
    \begin{equation*}
        \lambda = \partition{\lambda[:j]\\ \play{\lambda_{j+1}} \\ 0 \\ \vdots} = \partition{\alpha\\ \gamma^{(t)}} =  \partition{\alpha \\ \play{\sigma} \\ 0 \\ \vdots \\ 1}
    \end{equation*}
    where $\alpha = \lambda[:j]$ has $\alpha_j \ge 1$, with $\sigma \in \{0, 1\}$ and $\partition{\play{\sigma} \\ 0 \\ \vdots \\ 1} = \gamma^{(t)}$ chosen as in Claim~\ref{claim:special-tree}.  Then $S$ has growth function
    \begin{equation*}
        h(x) = A(x) + B(x)g_t(x)
    \end{equation*}
    where $A, B \in \mathbb{Z}[x]$ of degree at most $j$. 
\end{claim}

\begin{proof}
    If $\play{\sigma}$ is bracketed, any parts of $\alpha$ are bracketed, due to rule $3$. Since $\lambda_j = \alpha_j \ge 1 > = 0 = \lambda_{j+2}$ and $\sigma$ is the only playable part in the segment $\gamma^{(t)}$, by the proof of Proposition~\ref{prop:playable} and rule $1$, the part $\lambda_{j+2} = \gamma^{(t)}_2 = 0$ cannot be playable any time before $\sigma$ is played, and the same for any entries below it.  Let $\delta$ be a playing sequence starting at $\lambda$. Let $ind_\sigma^\delta(t)$ be the index of $\sigma$ in $R_{\delta[:(t-1)]}(\lambda)$.    We have two cases:
    \begin{itemize}
        \item[1.] Play $\lambda \xrightarrow{\delta} R_\delta(\lambda)$ where $\delta_i \le ind_\sigma^\delta(i)$ for all $i$. Clearly, $l(\delta) \le j$ because we delete the played entry in each operation, by rule $2$. Let $A(x)$ be the growth function for this set of elements obtained from performing such playing sequences.
        \item[2.] Part $\sigma$ is played at some time; assume $\lambda \xrightarrow{\delta} R_\delta(\lambda)$ where $\delta_i = ind_\sigma^\delta(i)$. Because $\alpha_j \ge 1$, the rule $3$ confirms that in any states before playing $R_{\delta_i}$, the immediate part above the tail segment $\gamma^{(t)}$ is at least $1$. Hence $R_{\delta[:i]}(\lambda)$ results in $\partition{\beta \\ \gamma^{(t+1)}}$, in which the lowest part of $\beta$ is at least $2$ and we claim that $\beta$ has $j-i+1$ parts.  The reason for this claim is because we erase one part of $\alpha$ each time we play $\delta[:(i-1)]$. Thus playing any parts of $\beta$ cannot make the $0$ at top of $\gamma^{(t+1)}$ playable due to rule $2$. From then, if any $\delta_{i+s} < \delta_i$ for some $s>0$, that is, a part of the segment $\beta$ is played, then any parts below $\beta$ are not playable, then the rest of $\delta$ only plays parts of $\beta$ and leads to a leaf. 
        
        The playing sequences $\delta$ in this case take $i-1$ levels to reach $\gamma^{(t)}$. From then, playing $[\ge \delta_i]$ forms a copy of $\gamma^{(t)}$, while playing entries of $\beta$ at any time reaches a leaf in at most $j - i + 1$ plays. Finally, these playing sequences contribute $B(x)g_t(x)$ for $B \in \mathbb{Z}[x]$ of degree $j$.  
    \end{itemize}
\end{proof}

\begin{proof}[Proof of \textbf{Theorem~\ref{thm:denom-deg}}]
We only need to connect the dots by proving that every branch of the tree $G_s$ rooted at $\gamma^{(s)}$  will eventually hit the element of the form $\lambda$ as in Claim~\ref{claim:limit-deg-function}. In $\mathcal{C}_P$,
one cycles the positions of the $1$'s to obtain the successive elements $\gamma^{(s)}$, so in any $G_s$ with $s \neq t$, the first part $\gamma^{(t)}_1$ is at some position $\gamma^{(s)}_{j+1}$ ($1 \le j \le p-1$). We can append some entries at the end of $\gamma^{(s)}$ so that the tail is $\gamma^{(t)}$. Hence the Claim~\ref{claim:limit-deg-function} applies, with $j$ at most $p-1$.

Furthermore, $g_t(x) = 1 + xg_{t+1}(x)$ because $\gamma^{(t)}_1$ is the only playable part in $\gamma^{(t)}$. Now there are polynomials $K, L$ of degree at most $p-1$ satisfying
\begin{equation*}
    g_{t+1}(x) = K(x) + L(x)g_t(x) = K(x) + L(x)(1 + xg_{t+1}(x)) = K(x) + L(x) + xL(x)g_{t+1}
\end{equation*}
implying there are polynomials $P, Q$ of degree at most $p$ such that
\begin{equation*}
    g_{t+1}(x) = \dfrac{M(x)}{N(x)}.
\end{equation*}
Then $g_t(x)$ is rational with denominator of degree at most $p$ and numerator of degree at most $p+1$. Specifically,
\begin{equation*}
    g_t(x) = \dfrac{N(x) + xM(x)}{N(x)}.
\end{equation*}
Any other $g_s(x)$ where $s \not \in \{t, t+1\}$ is of the form 
\begin{equation*}
    g_s(x) = A(x) + B(x)g_t(x) = \dfrac{AN+BN+xBM}{N},
\end{equation*}
where $A, B$ are polynomials of integer coefficients of degree at most $p-2$. Thus any such $g_s$ is rational with denominator of degree at most $p$ and numerator of degree at most $2p-1$. Let $g$ be the growth function of the $P$ quasi-infinite forest, then $g = \sum_{s=1}^p g_s$ is rational with denominator $Q(x)$ and numerator of degree at most $2p - 1$. 

Now we want to discard a copy of the forest coming from the branches $[1\ldots]$ in each tree $G_s$ to get the limit generating function $H_P$ of \rebs{P^\infty}. Thus 
\begin{equation*}
    g(x) = H_P(x) + xg(x)
\end{equation*}
and then 
\begin{equation*}
    H_P(x) = (1-x)g(x)
\end{equation*}
a rational function with denominator of degree at most $p$ and numerator of degree at most $2p$.

\end{proof}

\section{Further discussion and conjectures on the finite BS systems}

  The precise limiting generating functions $H_{BWW}$ and $H_{BBW}$ were computed in detail the author's bachelor thesis \cite[\S 5]{pham2022bs}. The limiting generating functions for primitive necklaces of greater length can be easily computed in that fashion. Below are some of them that were computed by hand:
\begin{align*}
    H_{BWW}(x) = H_{BBW}(x) & = (1-x)\dfrac{x^3-3x^2-4x-3}{2x^3+x^2-1},\\
    H_{BWWW}(x) &= (1-x)\dfrac{x^5+8x^4-3x^3-8x^2-6x-4}{6x^4+4x^3+x^2-1},\\
    H_{BBBW}(x) &= (1-x)\dfrac{2x^5+8x^4-5x^3-10x^2-7x-4}{6x^4+4x^3+x^2-1}, \\
    H_{BBWW}(x) &= (1-x)\dfrac{x^3+x^2+x+1}{3x^4+2x^3+x^2-1}, \\
    H_{BWWWW}(x) &= (1-x)\dfrac{2x^6+16x^5-12x^4-23x^3-16x^2-8x-5}{12x^5+8x^4+2x^3-1}.
\end{align*}
They led us to conjecture in addition to Theorem~\ref{thm:denom-deg} that the denominator degree of $H_P$ is exactly $\abs{P}$.

 Another interesting question about the Bulgarian Solitaire dynamical system is the sizes of the orbits, which are parametrized by necklaces as discussed in Section~\ref{sec:1}. Recall that if $N = P^k$ for some primitive necklaces $P$ of length $p$, an element $\lambda$ in the recurrent set $\bm{C}_{P^k}$ is of the form $\lambda = (\gamma^{(t)})^k + \Delta_{pk - 1}$ for some $\gamma^{(t)} \in \mathcal{C}_{P}$. Since the number of $1$'s in $\gamma^{(t)}$ is equal to the number of black beads in $N$, we have that the size of the partition $n$ that the Bulgarian Solitaire operation acts on is
\begin{equation*}
    n = \binom{pk}{2} + k \cdot \#\text{black beads of }P
\end{equation*}

As we know, when $N = W^k$, the orbit $\mathcal{O}_{W^k}$ is actually the whole Bulgarian Solitaire system on the partition set $\mathcal{P}\left( \binom{k}{2} \right)$. The author's bachelor thesis \cite{pham2022bs} computed the sizes of orbits parametrized by necklaces of the form $(BW)^k, (BBW)^k$ and $(BWW)^k$. One of the results involve the sequence of \emph{Chebyshev polynomials of the first kind} $\{T_k(x)\}_{k=0}^\infty$ evaluated at $x = 2$, which satisfies the recurrent formula
\begin{align*}
    T_0(2) &= 1 \\
    T_1(2) &= 2 \\
    T_{k}(2) &= 4T_{k-1}(2) - T_{k-2}(2) \qquad \text{for } k \ge 2.
\end{align*}
The results are the theorem below:
\begin{theorem}[Pham \cite{pham2022bs}]
    For each $k = 1, 2, \ldots$, one has
    \begin{align*}
        \abs{\mathcal{O}_{(BW)^k}} &= T_k(2) \\
        \abs{\mathcal{O}_{(BWW)^k}} &= 5^k \\
        \abs{\mathcal{O}_{(BBW)^k}} &= 7 \cdot 5^{k-1}
    \end{align*}
\end{theorem}

We also conjectured that orbits parametrized by $P^k$ for primitive necklaces $P$ of length greater than $3$ grow geometrically as well. Some data for primitive necklaces of length $4$ and $5$ are given in \cite[Section 3.1]{pham2022bs}.
\begin{conj}\label{conj:reccurent-orb-size}
    For any primitive necklace $P$ with $|P| \geq 3$,
    there is an integer $c_P$ such that
    for $k \geq 2$,
    \begin{equation*}
         \abs{\mathcal{O}_{P^k}} = (c_P)^{k-1} \abs{\mathcal{O}_P} 
    \end{equation*}
     Moreover, when $P$ and $P'$ are related by swapping black beads to white beads, then $c_P = c_{P'}.$
\end{conj}

\nocite{*}
\bibliographystyle{abbrv}

\end{document}